\documentclass[11pt, reqno]{amsart}

\usepackage[utf8]{inputenc}

\usepackage[letterpaper,hmargin=1.0in,vmargin=1.0in]{geometry} 
\usepackage{amsmath,amsthm,amsfonts,amssymb,mathrsfs,stmaryrd,bigints}
\usepackage{dsfont}
\usepackage{mathtools}
\usepackage[usenames]{color}
\usepackage{float}
\usepackage[colorlinks=true,linkcolor=blue]{hyperref}
\usepackage{tikz}
\usepackage{subcaption}
\usepackage{epstopdf} 
\usepackage{amssymb}
\usepackage{setspace}
\usepackage{enumerate}
\usepackage{bigstrut}
\usepackage{multirow}
\usepackage{mathtools,xparse}
\usepackage{mathrsfs}
\usepackage{hyperref}
\usepackage{xcolor}
\usepackage [autostyle, english = american]{csquotes}
\usepackage{pgfplots}
\newtheorem*{lemma*}{Lemma}
\newtheorem{theorem}{Theorem}[section]

\newtheorem{lemma}[theorem]{Lemma}
\newtheorem{proposition}[theorem]{Proposition}
\theoremstyle{definition}
\newtheorem{definition}{Definition}
\newtheorem{remark}[theorem]{Remark}

\allowdisplaybreaks
\usepackage{chngcntr}

\newcommand{\E}{{\mathbb{E}}}

\newcommand{\norm}[1]{\left\lVert#1\right\rVert}

\newcommand{\e}{\varepsilon}

	\renewcommand{\P}{\mathbb{P}}

\newcommand{\cC}{\mathcal{C}}
\newcommand{\cD}{\mathcal{D}}

\newcommand{\cG}{\mathcal{G}}

\newcommand{\cT}{\mathcal{T}}

\renewcommand{\setminus}{\backslash}

\newcommand{\TIG}{\mathsf{TIG}}
\newcommand{\TISG}{\mathsf{TISG}}

\def\ba{\begin{align}}
\def\ea{\end{align}}
\def\bs{\begin{split}}
\def\es{\end{split}}

\begin{document}

\title[Upper tail problem in the Poisson regime]{Upper tail behavior of the number of triangles in random graphs with constant average degree}

\author{Shirshendu Ganguly, Ella Hiesmayr, and Kyeongsik Nam}

\begin{abstract} 
Let $N$ be the number of triangles in an Erd\H{o}s–R\'enyi graph $\cG(n,p)$ on $n$ vertices with edge density $p=d/n,$ where $d>0$ is a fixed constant. It is well known that $N$  weakly converges to the  Poisson distribution with mean ${d^3}/{6}$ {as $n\rightarrow \infty$}. We address the upper tail problem for $N,$ namely, we investigate how fast $k$ must grow, so that $\P(N\ge k)$ is not well approximated anymore by the tail of the corresponding Poisson variable. Proving that the tail  exhibits a sharp phase transition, 
we essentially show that the upper tail is governed by Poisson behavior only when $k^{1/3} \log k< (\frac{3}{\sqrt{2}})^{2/3} \log n$  (sub-critical regime) as well as pin down the tail behavior when $k^{1/3} \log k> (\frac{3}{\sqrt{2}})^{2/3} \log n$ (super-critical regime). We further prove a structure theorem, showing that the sub-critical upper tail behavior is dictated by the appearance of almost $k$ {vertex-disjoint} triangles whereas in the supercritical regime, the excess triangles arise from a clique like structure of size approximately $(6k)^{1/3}$.  This settles the long-standing upper-tail problem in this case, answering a question of Aldous, complementing a long sequence of works, spanning multiple decades and culminating in \cite{upper_tail_localization}, which analyzed the problem only in the regime $p\gg \frac{1}{n}.$ The proofs rely on several novel graph theoretical results which could have other applications. 
\end{abstract}

\address{ Department of Statistics, Evans Hall, University of California, Berkeley, CA
94720, USA} 

\email{sganguly@berkeley.edu }

\address{ Department of Statistics, Evans Hall, University of California, Berkeley, CA
94720, USA} 

\email{ella.hiesmayr@berkeley.edu }

\address{ Department of Mathematical Sciences, KAIST, South Korea} 

\email{ksnam90@gmail.com }

\maketitle

\tableofcontents

\section{Introduction and main results} \label{section 1}

The statistical properties of the number of triangles or other local structures in random graphs have been a major topic of study for the last few decades witnessing several important advances.  Postponing a somewhat detailed overview of the literature to Section \ref{review}, let us first move towards stating the basic set up and the main results of this article.

Throughout the paper, let $G_n=\cG(n,\frac{d}{n})$ be the  Erd\H{o}s–R\'enyi graph on $n$ vertices with edge density $d/n$ where $d>0$ is a fixed constant. We will denote by $V_n:=V(G_n)$ and $E_n:=E(G_n),$ the corresponding vertex and edge sets respectively (note that $|V_n|=n).$
Further, let $N$ be the number of triangles in $G_n$, i.e., the number of unordered triples $(u,v,w)$ of vertices $u,v,w  \in V_n,$ where all the possible edges $(u,v),(v,w),(u,w)$ are elements of $E_n$. It is known \cite{bollobas1981} that $N$  weakly converges to the  Poisson distribution with mean $\frac{d^3}{6}$, to be denoted ${\mathsf{Poi}}(\frac{d^3}{6})$. This implies that, for any \emph{bounded} sequence $\{k_n\}_{n \in \mathbb{N}}$,
\begin{align} \label{poisson}
\Big\vert\mathbb{P}(N \geq k_n) -  \mathbb{P}\left({\mathsf{Poi}}\Big(\frac{d^3}{6}\Big) \geq k_n\right) \Big\vert \rightarrow 0
\end{align}
as $n\rightarrow \infty$, where throughout the paper, $\P$ and $\E$ will denote the underlying probability and the corresponding expectation.

The purpose of the article is to investigate at what depth into the tail, the Poisson behavior no longer holds, i.e.,
how fast  $\{k_n\}_{n \in \mathbb{N}}$ must grow such that no version of  \eqref{poisson} is true,  and further, how
does $G_n$ look like, conditional on $N \ge k$.

\subsection{Main results}
Throughout the paper, we assume that $k_n$ is increasing in $n$ and for the sake of simplicity, we will suppress the $n$ dependence and just use $k.$  
\begin{theorem} \label{theorem 1}
The following describes the ``upper-tail" probabilities of $N$.

1. There exists a constant $c>0$ depending only on $d$ such that   the following holds.  Let $\delta>0$ be a sufficiently small   constant. Suppose that
 $k^{1/3} \log k< ((\frac{3}{\sqrt{2}})^{2/3} - \delta) \log n$. Then, for sufficiently large  $n$,
\begin{align}\label{disj1}
 e^{-k\log k - ck} \leq \mathbb{P}( N  \geq  k) \leq e^{-k\log k + ck}.
\end{align}

2. {Let $\delta > 0$ and $0<\mu<\frac{1}{10}$ be constants.  Suppose that  $k^{1/3} \log k> ((\frac{3}{\sqrt{2}})^{2/3} + \delta) \log n$ and  $k\leq n^{\frac{1}{10}- \mu}$.  Then, for any $\e>0$ and sufficiently large  $n$,} 
\begin{align}\label{cliq}
n^{-((\frac{3}{\sqrt{2}})^{2/3}+\e) k^{2/3}}  \leq   \mathbb{P}( N  \geq  k) \leq n^{-((\frac{3}{\sqrt{2}})^{2/3}-\e) k^{2/3}}.
\end{align}
\end{theorem}
Thus Theorem \ref{theorem 1} establishes a tight bound on the upper tail behavior of $N$ and provides a precise critical location at which the same changes.
It shows that the Poisson tail holds when $k^{1/3} \log k\le  (\frac{3}{\sqrt{2}})^{2/3} \log n$ (to be called the sub-critical regime from now on) and not beyond, the super-critical regime.

The next result answers a question of David Aldous, who had asked about the structure of the graph $G_n,$ conditioned on $N\ge k.$ 
We start with the following notation. 
{For a graph $G = (V,E),$ let $\Delta(G)$ (simply $\Delta$ when $G$ is clear) be the number of triangles in $G$ and for any subset $W \subseteq V$, define $\Delta_G(W)$ to be the number of triangles in $G$ consisting of vertices in $W$.   If it is clear what the underlying graph $G$ is, then we suppress the $G$ dependence in $\Delta_G(W)$ and just use the notation $\Delta(W)$. }

\begin{theorem}[Structure theorem] \label{theorem 2}
Let $0<\mu<\frac{1}{10}$ be fixed and $k\leq n^{\frac{1}{10}- \mu}.$ 
Now for any  constant { $0<\varepsilon<1$} ,  define the following two events {
\begin{align*}
\mathcal{D}_{\e}:=\{  \text{There exist at least}  \    (1-\varepsilon) k \  \text{vertex-disjoint triangles} \}
\end{align*}
and
\begin{align*}
\mathcal{C}_{\e}:= \{ \text{There exists} \ V'\subseteq V_n  \  \text{such that} \ |V'| \leq   (1+ \varepsilon) 6^{1/3}k^{1/3}, \Delta(V') \geq  (1-\varepsilon ) k \}.
\end{align*}}
In words, $\mathcal{C}_{\e}$ denotes the presence of an ``almost" clique accounting for most triangles. 

Then,  
\begin{align} \label{011} 
\lim_{n\rightarrow \infty} \mathbb{P}( \mathcal{D}_\e \cup \mathcal{C}_\e | N \geq k)  = 1.
\end{align}
This, in conjunction with Theorem \ref{theorem 1}, implies  the following structure theorem.  Let $\delta>0$ be a sufficiently small constant. If  $k^{1/3} \log k< ((\frac{3}{\sqrt{2}})^{2/3} - \delta) \log n$, then
\begin{align}  \label{012}
\lim_{n\rightarrow \infty} \mathbb{P}(\mathcal{D}_\e | N \geq k)  = 1,
\end{align}
and if  $k^{1/3} \log k> ((\frac{3}{\sqrt{2}})^{2/3} + \delta) \log n$, then
\begin{align}  \label{013}
\lim_{n\rightarrow \infty} \mathbb{P}(\mathcal{C}_\e | N\geq k)  = 1.
\end{align}
\end{theorem}

Thus, the above states that in the sub-critical regime, it is likely that, conditional on $\{N\ge k\}$, one sees the appearance of almost $k$ {disjoint} triangles, while beyond that, an almost clique structure accounts for most of the excess triangles.

\begin{remark}\label{supercrit}
A few remarks are in order at this point. Note that the results for the super-critical regime in the above theorems 
essentially assume $k\le n^{1/10},$ whereas it should ideally work for $k=O(n^{3}).$ This technical condition is an artifact of our proof and allows certain union bound arguments to work. While it is possible to push this with a more efficient union bound scheme, for expository purposes, we made no such attempt, considering the central problem we were aiming to solve was to find the threshold until which the Poisson statistics remain valid.

{However, during the writing of the paper, a preprint \cite{cvh} was posted on the arXiv, which proves (see Theorem  1.2  therein), among other things, one half of our Theorem \ref{theorem 1} for a slightly different range of $k$. In particular, it is shown that the supercritical behavior, i.e., \eqref{cliq}, holds for $k$ growing faster than  $(\log n)^3$ (see \cite[Remark 1.3]{cvh} for details).  Note that the critical location in Theorem \ref{theorem 1} is given by $k^{1/3}\log k \approx (\frac{3}{\sqrt{2}})^{2/3}\log n$ \footnote{$\approx$ will be informally used throughout the article to denote 	`close to' in a sense whose meaning might change across locations depending on the context. We will refrain from being more precise.} which implies $k \approx (\frac{\log n}{\log \log n})^3 \ll (\log n)^3$, and so the result in \cite{cvh} falls short of going down to the critical location and further does not provide any insight into the subcritical behavior. However, on the other hand, unlike our technical condition, which is essentially $k\leq n^{\frac{1}{10}}$, there is no restriction on the upper bound for $k$ in their results.  It is worth mentioning that  the methods in \cite{cvh} are rather different from ours and rely more on refining the approach in \cite{upper_tail_localization}.}

Summarizing, while our result covers the sub-critical region and the super-critical region up to $k\le n^{1/10}$, the result in \cite{cvh} covers the super-critical region from $(\log n)^3$ onwards. 

\begin{center}
\textbf{Thus, the two results together settle the long-standing ``upper tail problem" for the number of triangles in the case where the average degree is a constant. }
\end{center}
\end{remark}

Before discussing the key proof ideas, we now provide a broad overview of the advances made on the upper tail problem in various settings over the years. 

\subsection{Previous work}  \label{review}

While the present paper considers the Erd\H{o}s-R\'enyi 
random graph with constant average degree, the problem 
was initially considered for a 
general  density $p,$ i.e., for 
the random graph $\cG({n, p}).$ 
While, eventually, counts for various subgraphs were
 studied quite closely, the initial works in this direction
  tried to pin down the tail probabilities for $N,$ the
   triangle count, as in the present paper. Thus, formally,
  the upper tail problem for $N$ asks to
    estimate the large deviation rate function given by
\[ R(n,p,\delta) := -\log \P\left(N \geq (1+\delta)\E[N]\right)\quad\mbox{ for fixed $\delta > 0$}\,.\] 
This simple to state problem turned out to be fundamental and extremely challenging, leading to intense research for over two decades, 
(see e.g., ~\cite{JR02,Vu01,KV04,JOR04,JR04,Cha12,DK12b,DK12} and~\cite{Bol,JLR} and the references therein). It followed from the works \cite{Vu01},\cite{KV04} that
\footnote{We write $f \lesssim g$ to denote $f = O(g)$; $f \asymp g$ means $f = \Theta(g)$;  $f \sim g$ means $f = (1+o(1))g$ and $f \ll g$ means $f = o(g)$.}
\[   n^2 p^2 \lesssim R(n,p,\delta)  \lesssim n^2 p^2\log(1/p)\]
(the harder direction, the lower bound, relied on the  ``polynomial
concentration'' machinery, 
whereas the 
upper bound is obtained by planting a clique, 
and observing that a set of $s \sim  \delta^{1/3} np$ vertices can form a clique with
  probability $p^{\binom{s}2}=p^{O(n^2 p^2)}$, thus contributing roughly $\delta \binom{n}{3} p^3 $ 
  extra triangles). 
  In parallel work, Chatterjee~\cite{Cha12}, and  DeMarco and Kahn~\cite{DK12} settled the sharp order of the rate function showing 
\begin{equation}\label{meanfield}
R(n,p,\delta) \asymp n^2 p^2 \log(1/p)\quad \text{ for }  p \ge \frac{\log n} n.
\end{equation} 
However, the methods in \cite{Cha12,DK12,DK12b} were not strong enough 
to recover the exact asymptotics of this rate
  function leading one to wonder if planting a clique is always the best strategy. 
  

Progress  in this front has subsequently witnessed an explosion kickstarted by  
the pioneering work of Chatterjee and
  Varadhan~\cite{CV11} that introduced a large
   deviation framework for $\cG_{n,p}$ in the
    \emph{dense regime} ($0<p<1$ fixed) via
     the theory
     of graphons  (see the 
     survey by Chatterjee~\cite{Cha16}). However, the above framework is not equipped to handle the  sparse
  regime ($p\to 0$), where the understanding still  remained rather
     limited until another breakthrough by Chatterjee and
      Dembo \cite{CD16}. This reduced it to a natural
       variational problem in a certain range of $p$. 
       A significant amount of activity was then devoted to, at least asymptotically, solving the variational problem. 
        This was initiated in~
        \cite{LZ-sparse} (who also considered the case
         when the triangle was replaced by a general
          clique), thereby yielding the following
           conclusion: for fixed $\delta > 0$, if $n^{-1/42} \log n \leq p = o(1)$, then
\begin{equation} \label{eq:tail-prob-K3}
R(n,p,\delta) \sim I(\delta) n^2p^2 \log(1/p)\quad\mbox{ where }\quad I(\delta)=\min \bigl\{ \tfrac12 \delta^{2/3}, \tfrac13 \delta\bigr\}\,.
\end{equation}
Thus the aforementioned clique construction
 gives the correct leading order constant if $\delta \ge 27/8$.
This was subsequently extended to general subgraphs
 beyond the case of the clique in \cite{bglz}.

However, as was indicated in \eqref{meanfield} proved in \cite{Cha12,DK12}, \eqref{eq:tail-prob-K3} is expected to hold until the much lower threshold of $p\gg \frac{\log n}{n}.$ Several important attempts were made subsequently in this direction. 
\cite{augeri1,cookdembo} pushed the result down to $p\gg \frac{1}{\sqrt n},$ where the problem undergoes a natural transition. More recently, in a breakthrough paper \cite{upper_tail_localization}, using a different approach going back to the classical work of Janson, Oleszkiewicz and
Ruci\'nski \cite{JOR04}, the upper tail problem was essentially completely solved in the regime $p\gg \frac{1}{n}.$ This includes in particular the window $\frac{1}{n}\ll p\ll \frac{\log n}{n}$ beyond the `mean-field' regime covered in \eqref{meanfield}.
It is known that in the latter regime, the triangle count $N$ behaves as a Poisson variable. 
Further, it is straightforward to verify that for a Poisson variable $X$ with mean $\mu,$
$$-\log [\P(X\ge (1+\delta)\mu)]=(1+o(1))\big((1+\delta)\log(1+\delta)-\delta \big)\mu $$
as $\mu \to \infty$. 
Now, when $\frac{1}{n}\ll p\ll \frac{\log n}{n},$ the mean triangle count $\E(N)$ does diverge to infinity as $n\to \infty$ falling in the above setup and it was indeed shown in \cite{upper_tail_localization} that in this case,
$$-\log [\P(N\ge (1+\delta)\E(N))]=(1+o(1))\big((1+\delta)\log(1+\delta)-\delta \big)\E(N),$$
(for a more detailed statement and related results, the reader is encouraged to refer to \cite[Section 8]{upper_tail_localization}).

However, none of the methods described above are equipped to handle the case of constant average degree, where $N$ behaves like a Poisson random variable with bounded mean, the object of study in the present paper.

Finally, it might be worth pointing out that while the previous works were only looking at the large deviation regime, i.e., deviation from the mean by a constant multiplicative factor, a regime which is not interesting when the mean is bounded, we focus on the ``entire" tail of the random variable $N.$

\subsection{Idea of proof}\label{iop} In this section, we highlight some of the key ideas going into our arguments. Broadly, there are two distinct kinds of ingredients that go into the proofs:  probabilistic and graph theoretical ones. 
While the probabilistic ideas mostly pertain to  the results of this article, the graph theoretical results could be of independent interest and find other applications. 

To get started, we notice that there are two natural ways for a  graph  to have $k$ triangles: either possessing $k$ vertex-disjoint triangles or a clique of size approximately $(6k)^{1/3}$.  Using a standard second moment bound, one can obtain rather sharp estimates on  the probability of  such events, which yields a lower bound for the probability $\mathbb{P}(N\geq k)$.

Most of the article is devoted to the highly delicate task of obtaining matching upper bounds and thereby establishing that these are the `only' ways,  {probabilistically speaking.}

 To accomplish this, on the event that $G_n$ has $k$ triangles, we focus on the subgraph induced by the  triangles and analyze the connected components. We call such a component a triangle induced graph, more generally, a graph $G=(V,E)$ is called
 a triangle-induced graph (to be called a $\TIG$ from here on)  if it is  a connected graph and can be obtained by taking a union of triangles (see Definition \ref{def}).
 
The main goal is to now obtain a sharp upper bound on the probability that $G_n$ contains such a $\TIG$. Equipped with the same, the van den Berg-Kesten (BK) inequality  {\cite{inequality_II}}, (which bounds the probability of disjoint occurrences of several events), then allows us to bound the probability of several disjoint {$\TIG$s}, thereby establishing the desired upper bound for  $\mathbb{P}(N\geq k)$.
\noindent
More precisely, note that  the event $\{N\geq k\}$ implies  the occurrence  of  vertex-disjoint $\TIG$s  each induced by $\ell_i$, say for $i=1,\cdots,m$, triangles with $1=\ell_1=\cdots =\ell_j < \ell_{j+1} \leq \cdots \leq \ell_m$ and 
$\ell_1+\cdots+\ell_m =  k$ (we denote this event by $E_{\ell_1,\cdots,\ell_m}$).  Also, let $F_\ell$ be the event that there exists a subgraph which is a $\TIG$ induced by $\ell$ triangles.  Then, by  BK inequality,
\begin{align}  \label{bk}
\mathbb{P}(E_{\ell_1,\cdots,\ell_m}) \leq  \mathbb{P}(E_{1,\cdots,1})  \mathbb{P}(F_{\ell_{j+1}}) \cdots   \mathbb{P}(F_{\ell_m})
\end{align}
(there are $j$ many 1s indexing $E_{1,\cdots,1}$).

First,  it is straightforward to upper bound the probability $\mathbb{P}(E_{1,\cdots,1})$  (i.e. existence of $j$ vertex-disjoint triangles). It turns out that the upper bound obtained by a naive first moment method    is indeed sharp, which can be established by computing its second moment:
\begin{align} \label{001}
\mathbb{P}(E_{1,\cdots,1}) \approx e^{-j\log j}.
\end{align}

Most of the new ideas go into the proof of the upper bound of $\mathbb{P}(F_\ell)$, which is one of our key results. In other words, we estimate the probability that  there exists a  connected subgraph, say $H$, induced by $\ell$ triangles. If $H$ has $v$ vertices and  $e$ edges, then the probability that $G_n$  contains  such an $H$ is bounded by
\begin{align*}
n^v v^{2e}   \cdot \Big(\frac{d}{n}\Big)^e = v^{2e}d^e \frac{1}{n^{e-v}},
\end{align*}
since the number of such subgraphs in $G_n$ is bounded by $n^v {v \choose 2}^e \leq n^v v^{2e}$ (the $n^{v}$ term coming from the choices of the vertices, while ${v \choose 2}^e$ bounds the number of choices for edges).\\

The proof now hinges on a crucial graph-theoretical ingredient which we prove in Section \ref{section 5}, which states that $e-v$ can be lower bounded in terms of the number of triangles. In other words, we deduce that there exists an `almost' (we will not make this notion precise in this discussion) concave function $h$ with $h(y) \approx ay^{2/3}$ ($a$ is an explicit constant) such that
\begin{align*}
e-v \geq h(\ell)
\end{align*}
(The function $h$ is essentially the inverse of the function appearing on the RHS of the inequality in Lemma  \ref{lemma 3.2}. Several analytic properties of the function $h$ which are exploited crucially are established in Section \ref{convexsec}.)
The next step implements an efficient union bound scheme over all possible subgraphs $H$ induced by $\ell$ triangles. Without going into details, let us just mention that this step relies on establishing that any such $H$ must admit a further `dense' subgraph (where the number of edges is proportional to the square of the number of vertices) which accounts for most of the triangles. 

Putting the above together, we essentially establish that
\begin{align} \label{002}
\mathbb{P}(F_\ell) \leq   \Big(\frac{1}{n}\Big)^{ h(\ell) }
\end{align} 
(see Lemma  \ref{lemma 3.3} for a precise statement).
It might be worth mentioning that the recent preprints \cite{andreis21homo,andreis21} study the somewhat related notion of  large deviation properties of the connected components in sparse random graphs.

Therefore,  applying \eqref{001} and \eqref{002} to \eqref{bk}, we approximately have an upper bound
\begin{align} \label{004}
\mathbb{P}(E_{\ell_1,\cdots,\ell_m}) \leq  e^{-j\log j}   \Big(\frac{1}{n}\Big)^{h(\ell_{j+1} )+ \cdots  + h(\ell_m) }.
\end{align}
By concavity properties of $h$, this quantity  is  bounded by
\begin{align} \label{003}
 e^{-j\log j}   \Big(\frac{1}{n}\Big)^{h(\ell_{j+1} +\cdots+\ell_m) } =  e^{-j\log j } \Big(\frac{1}{n}\Big)^{h(k-j)} \approx  e^{-j\log j } \Big(\frac{1}{n}\Big)^{ a(k-j)^{2/3} }.
\end{align}
where the last approximate equality is obtained using the aforementioned asymptotics of the function $h.$

Towards analyzing the bound in \eqref{003}, let us define   a function $\phi(j) :=  e^{-j\log j } n^{ -a(k-j)^{2/3} }$. Then, by the above,
\begin{align} \label{005}
\mathbb{P}(E_{\ell_1,\cdots,\ell_m}) \leq  \phi(j)
\end{align}
(recall that there are $j$ 1s in $\ell_1,\cdots,\ell_m$).
It turns out, and is not difficult to check, that the  function $\phi$ exhibits the following transition.
\begin{align} \label{006}
\begin{cases}
k^{1/3}\log k < a\log n \  \Rightarrow  \ \phi(j) \ \text{attains the maximum at} \ j\approx k \\
k^{1/3}\log k > a\log n \  \Rightarrow \  \phi(j) \ \text{attains the maximum at} \ j\approx 0.
\end{cases}
\end{align}
Using \eqref{005}, we obtain an upper bound for  $\mathbb{P}(E_{\ell_1,\cdots,\ell_m})  $ by taking a union bound over  all possible tuples $(\ell_1,\cdots,\ell_m)$  such that $\ell_1\leq \cdots \leq \ell_m$ and  $\ell_1+\cdots+\ell_m =  k$. Since by \cite{partition} the number of such partitions grows like $e^{\sqrt{k}},$ essentially the bound in \eqref{005} prevails.

The aforementioned argument also allows us to prove Theorem \ref{theorem 2}.
In the sub-critical regime, by \eqref{006}, conditioned on possessing $k$ triangles, it is likely  to have `almost' $k$ vertex-disjoint triangles, whereas, in the super-critical regime, with high probability, $j\approx 0$ and thus $\ell_{j+1} + \cdots + \ell_m \approx k$. Further,  the application of Jensen's inequality in \eqref{003} is sharp if and only if $\ell_m \approx k$ and other $\ell_i$s are negligible.  This implies that  with high probability conditioned on $\{N\geq k\}$, there exists a $\TIG$ induced by $k'\approx k$ triangles.

At this point another crucial graph theoretic result is proven. Namely, that  conditioned on a $\TIG$ having $k'$ triangles, it is likely that this component contains an approximate clique accounting for almost $k'$ triangles. We will not elaborate on the proof in this discussion beyond mentioning that this involves proving a structure theorem quantifying when the following well known inequality (known as Kruskal-Katona bound, see e.g. \cite{alon}), bounding the number of triangles in any graph in terms of the number of edges,
$$\Delta \leq \frac{\sqrt{2}}{3}|E|^{3/2},$$  is almost sharp ({it is a straightforward computation to check that, ignoring lower order terms, equality holds when the graph is a clique}).

Though it has been pointed out to us, after the completion of the paper, by Noga Alon that a similar structural result had already been proved quite a few years back in \cite{keevash08}, we decided to keep the proof, since it seems to be somewhat different from the previous one, borrows ideas from spectral graph theory and could be of independent interest. 

\subsection{Organization of paper}
{The rest of the paper is organized as follows. In Section \ref{seclower},  we deduce lower bounds for $\mathbb{P}(N \geq k)$. In Section \ref{section 5}, we introduce and prove several  graph theoretical results that have been already alluded to in Section \ref{iop}.  In Section  \ref{convexsec},  we  establish a sharp upper bound on the probability that  $G_n$ contains a subgraph which is a $\TIG$ induced by $\ell$ triangles (i.e. a precise version of \eqref{002}). Using this combined with a technical convexity argument,  in Section \ref{secupper} we proceed with the details of \eqref{004}-\eqref{006} and deduce a matching upper bound on $\mathbb{P}(N \geq k)$. Finally,  in Section \ref{secstructure}, we prove Theorem \ref{theorem 2}.}

Before embarking on the proofs, we remark that, adopting standard practice, often in the proofs, we will use the same constant, say $C$, whose value might change from line to line.

\subsection{Acknowledgement}
S.G. thanks Sourav Chatterjee for mentioning to him the question of studying the upper tail of $N$, which he had learnt from David Aldous. S.G. also thanks Noga Alon for pointing him to \cite{keevash08}. His research was partially supported by NSF grant DMS-1855688, NSF CAREER grant DMS-1945172, and a Sloan Fellowship. E.H. was partially supported by NSF grant DMS-1855688.

\section{Lower bound}\label{seclower}
We start by establishing lower bounds for the tail probability of $N$ by analyzing two specific events, namely having $k$ disjoint triangles and having a clique of size approximately {$(6k)^{1/3}.$}
Let us introduce the following parameters which will recur throughout the article,
\begin{align} \label{d}
d_0:= \frac{d^3}{6},\qquad d_1:= \frac{ed^3}{6}.
\end{align}
To get started, let $H_k$ be the graph of size $3k$ vertices formed by $k$ disjoint triangles.

\begin{lemma}  \label{lower1}
There is a constant $c>0$ {depending only on $d$}, such that the following holds.
Suppose that $k\rightarrow \infty$ and $\frac{k^2}{n}\rightarrow 0$ as $n\rightarrow\infty$.   Then,  for large enough $n$,
\begin{align*}
\mathbb{P} (G_n \  \textup{contains}  \  H_k ) \geq  e^{-k\log k  - ck}.
\end{align*}
\end{lemma}

\begin{proof}

Let $X$ be the number of subgraphs $H_k$ in $G_n,$ which in addition satisfy that vertices in $H_k$ have no other edges. In other words,
{
\begin{align*}
X = \sum_{T_1,\cdots,T_k \text{ disjoint triangles}} \mathbf{1} ( & \text{edges are  present in} \  T_1, \dots, T_k,  \\
&   \hspace{1cm} \text{all vertices in} \  T_1\cup \dots \cup T_k \  \text{have no other edges}).
\end{align*} 
}
Note that there are $ \frac{1}{k!} {n \choose 3} {n-3 \choose 3} \cdots {n-3k+3 \choose 3}$ terms in the summation. 
Thus, for any $\e>0$, for  large enough $k$, 
\begin{align} \label{first1}
\mathbb{E} X & =  \frac{1}{k!} {n \choose 3} {n-3 \choose 3} \cdots {n-3k+3 \choose 3}    \Big(\frac{d}{n}\Big)^{3k} \Big(1-\frac{d}{n}\Big)^{3k(n-3k) + \frac{3k(3k-1)}{2} - 3k }   \\
& = {\frac{1}{k!} \Big(\frac{d^3}{6}\Big)^k   \frac{n(n-1)\cdots (n-3k+1)}{n^{3k}} \Big(1-\frac{d}{n}\Big)^{3k(n-3k) + \frac{3k(3k-1)}{2} - 3k } }  \nonumber \\
& \geq {\frac{1}{k!} \Big(\frac{d^3}{6}\Big)^k  \Big(1-\frac{3k}{n}\Big)^{3k} \Big(1-\frac{d}{n}\Big)^{3kn }} \geq  \frac{1}{k!} \Big(\frac{d^3}{6}\Big)^k e^{-(3+\e) dk}. \nonumber
\end{align}
{(We used $\lim_{n\rightarrow \infty} (1-\frac{1}{n})^n = e^{-1}$ and  the condition   $\frac{k^2}{n}\rightarrow 0$ to deduce the last inequality.)}
 Recall that by Stirling's formula,
 \begin{align}\label{stirling}
 \Big(\frac{k}{e}\Big)^k <  k!  <  C \sqrt{k}\Big (\frac{k}{e}\Big)^k
 \end{align}
 Thus, further using that $d_1=\frac{ed^3}{6}$ from \eqref{d},
\begin{align} \label{first}
\mathbb{E}X \geq  C \frac{1}{\sqrt{k}} \frac{d_1^k}{k^k}e^{-(3+\e) dk} = C e^{-k\log k + (\log d_1- (3+\e)d ) k - \frac{1}{2}\log k}.
\end{align}

Let us now compute the second moment of $X$.  {Since we require in the definition of $X$ that the vertices of the $k$ triangles have no other edges, computing the second moment of $X$  involves considering the following situation: two sets of $k$  triangles share some triangles and are otherwise vertex-disjoint.  The contribution from such pairs of subgraphs $H_k$s} sharing
 $k-\ell$ triangles for $\ell = 0,1,\cdots,k$ is given by
\begin{align*}
\sum_{\ell=0}^k &  \frac{1}{(k+\ell)!} {n \choose 3} {n-3 \choose 3} \cdots {n-3(k+\ell)+3 \choose 3}  {k+\ell \choose k} {k \choose \ell}     \\
&\cdot  \Big(\frac{d}{n}\Big)^{3(k+\ell)  } \Big(1-\frac{d}{n}\Big)^{3(k+\ell)(n-3(k+\ell)) + \frac{3(k+\ell)(3(k+\ell)-1)}{2} - 3(k+\ell)  }   .
\end{align*}
Above, the product of the 
$k+\ell$ many binomial ${n-3i \choose 3}$ terms, along with the $\frac{1}{(k+\ell)!}$, comes from choosing $k+\ell$ disjoint triangles. The $ {k+\ell \choose k} {k \choose \ell}  $ term is exactly the number of possible ways to choose  two subgraphs $H_k$s sharing $k-\ell$ triangles. Since we require that vertices in $k+\ell$ triangles cannot have other edges, we obtain the last two terms.

Note that the $\ell=0$ term is simply $\mathbb{E}X$. To bound each term,  since $\ell\leq k$ and $\frac{k^2}{n}\rightarrow 0$,  for any $\e>0$,  for large enough $n$,
\begin{align*}
\Big(1-\frac{d}{n}\Big)^{3(k+\ell)(n-3(k+\ell)) + \frac{3(k+\ell)(3(k+\ell)-1)}{2} - 3(k+\ell)  } \leq \Big(1-\frac{d}{n}\Big)^{3(k+\ell)n - O(k^2)} \leq    e^{-(3-\e) d(k+\ell)}.
\end{align*}
Also, note that $ {k+\ell \choose k} \leq 2^{k+\ell}$ and $   {k \choose \ell}  \leq 2^k$. 
Thus, for $\ell \geq 1$, using Stirling's formula \eqref{stirling},
 each term in the summation above is bounded by
 \begin{align*}
   \frac{1}{(k+\ell)!} \Big(\frac{d^3}{6}\Big)^{k+\ell}  e^{-(3-\e) d(k+\ell)} 2^{k+\ell} 2^k \leq  \Big(\frac{d_1}{k+\ell}\Big)^{k+\ell}   e^{-(3-\e) d(k+\ell)} 2^{2k+\ell}  
\end{align*}
(recall that $d_1 = \frac{ed^3}{6}$).
 Using \eqref{first}, the above quantity  is bounded by $  a_\ell \mathbb{E}X$,
where
\begin{align*}
a_\ell & = \Big  (\frac{d_1}{k+\ell}\Big)^{k+\ell}   e^{-(3-\e) d(k+\ell)} 2^{2k+\ell}\Big  (C \frac{1}{\sqrt{k}} \frac{d_1^k}{k^k}e^{-(3+\e) dk} \Big )^{-1}  \\
&= C^{-1} \frac{k^k}{(k+\ell)^{k+\ell}} d_1^\ell  e^{2\e dk} e^{-(3-\e)d\ell} 2^{2k+\ell}  \sqrt{k} \\
& \leq C  \frac{1}{(k+\ell)^\ell}    (d_1 e^{2\e d} 2^2 )^{k+\ell}  \sqrt{k}  \leq   C^{k+\ell} \frac{ \sqrt{k} }{(k+\ell)^\ell}  \leq  C^k \Big  (\frac{C}{k+\ell} \Big )^\ell
\end{align*}
($C>1$ is a constant depending on $d$ whose value above changes from line to line). Thus, for a large enough $k$,
\begin{align*}
\sum_{\ell=0}^k a_\ell    \leq  (k+1) C^k <C^{2k}.
\end{align*}
Hence,
\begin{align*}
\mathbb{E}X^2 \leq C^{2k} \mathbb{E}X. 
\end{align*} 
Therefore, since
\begin{align*}
\mathbb{P}(X \geq 1) \geq \frac{(E X)^2}{EX^2 }\geq  C^{-2k} \mathbb{E}X
\end{align*}
combined with \eqref{first}, we conclude the proof.

\end{proof}

{
\begin{remark} \label{triangle upper}
For later purposes we record an analogous upper bound using the first moment method. The expectation of the number of subgraphs $H_k$ is  
\begin{align} \label{first2}
\frac{1}{k!} {n \choose 3} {n-3 \choose 3} \cdots {n-3k+3 \choose 3}    \Big(\frac{d}{n}\Big)^{3k}  
\end{align}
Note that this is nothing other than the quantity  \eqref{first1}    in the previous lemma, without the power of $1-\frac{d}{n}$  term. This is because we no longer require vertices in $H_k$ to have no other edges. The quantity \eqref{first2}
is bounded by
\begin{align*}
 \frac{1}{k!} \Big(\frac{d^3}{6} \Big )^k   \frac{n(n-1)\cdots (n-3k+1)}{n^{3k}} \leq  \frac{1}{k!}\Big(\frac{d^3}{6}\Big )^k \leq e^{-k\log k + c'k}
\end{align*}
for some constant $c'>0$. Hence,  
\begin{align}
\mathbb{P} (G_n \  \textup{contains}  \  H_k ) \leq  e^{-k\log k  + c'k}.
\end{align}

\end{remark}
}

The second lower bound is obtained by planting a clique.
For a positive integer $k$, let $K_k$ be a clique of size $k$. The following lemma is rather straightforward and we omit its proof.

\begin{lemma} \label{lower2}

For any $n$ and $k$,
 \begin{align*}
\mathbb{P} (G_n \  \textup{contains}  \  K_k ) \ \geq  \Big(\frac{d}{n}\Big)^{ {k \choose 2} } .
\end{align*}
In particular,  for any $\e>0$,  for large enough $n$ and $k$, 
\begin{align*}
\mathbb{P}(G_n \  \textup{contains a clique of  size at least}   \ (6k)^{1/3}    ) \geq  n^{ -(1+\e)  (\frac{3}{\sqrt{2}})^{2/3} k^{2/3}}.
\end{align*}
\end{lemma} 

Therefore, combined with Lemma \ref{lower1}, we have the following result.

\begin{proposition} \label{lower}
There is a constant $c>0$, {depending only on $d$}, such that the following holds.
Suppose that $k\rightarrow \infty$ and $\frac{k^2}{n}\rightarrow 0$ as $n\rightarrow\infty$.  Then, for any $\e>0$, for large enough $n$,
\begin{align*}
\mathbb{P}(N \geq k ) \geq \max(e^{-k\log k - ck} ,  n^{-(1+\e)  (\frac{3}{\sqrt{2}})^{2/3} k^{2/3}} ).
\end{align*}
\end{proposition}

A simple calculation now shows that 
if $k^{1/3} \log k< ((\frac{3}{\sqrt{2}})^{2/3} - \delta) \log n$, then
\begin{align*}
e^{-k\log k - ck}>n^{-(1+\e)  (\frac{3}{\sqrt{2}})^{2/3} k^{2/3}},
\end{align*}
whereas,
if $k^{1/3} \log k> ((\frac{3}{\sqrt{2}})^{2/3} +\delta) \log n$ instead, then for small enough $\e>0$,
\begin{align*}
e^{-k\log k - ck}<n^{-(1+\e)  (\frac{3}{\sqrt{2}})^{2/3} k^{2/3}} .
\end{align*}
 
\section{Graph theoretic results}
 \label{section 5}
 As indicated in Section \ref{iop}, the proofs throughout the  article relies on key graph theoretic results, some of which we believe to be of independent interest, potentially having other applications.  
In this section, we record and prove all such results.

{From now on, for any subgraph $H$ of $G$, define $V(H)$ and $E(H)$ to be the set of vertices and edges of $H$ respectively.}  
Recall also that $\Delta(G)$ denotes the number of triangles  in $G$. To alleviate the notation, we simply use the notation $\Delta$ when $G$ is clear from the context.
Also, for any graph (which will be clear from the context), we denote the corresponding graph metric by $d(\cdot,\cdot).$

We start with  the following well known Kruskal-Katona bound on the number of triangles in a graph in terms of the number of edges. The proof is an application of  H\"older's inequality.  
\begin{lemma}\cite[Lemma 2.2]{bdm},\cite{alon} \label{lemmaholder}
For any graph $G=(V,E)$,
\begin{align}\label{holder}
\Delta \leq \frac{\sqrt{2}}{3}|E|^{3/2}.
\end{align}
\end{lemma}

We next restate the following crucial definition from Section \ref{section 1}.
 
 \begin{definition} \label{def}
A graph $G=(V,E)$ is called
 a triangle-induced graph (to be called a $\TIG$ from here on)  if it is  a connected graph and can be obtained by taking a union of triangles. 
 We say that a $\TIG$  $G=(V,E)$ is spanned by $\ell$ triangles if  there exist {(distinct)} triangles $T_1,\cdots,T_\ell$ such that {$V = \cup_i V(T_i) $ and $E = \cup_i E(T_i)$}.
\end{definition}

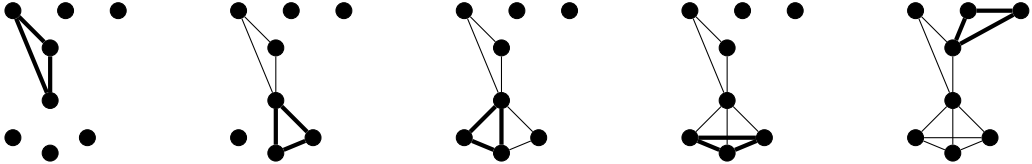
\begin{figure}[H]
\captionsetup{font=small}
\begin{tikzpicture}[node distance=0.7cm, scale = 1, font = \tiny]
\tikzstyle{every node}=[draw, shape=circle, minimum size = 0.1cm, inner sep=0, fill = black]

  \node (1) {1};
  \node (2)  [below right of=1] {2};
  \node (3) [below of =2] {3};
  \node (4) [below of =3] {4};
  \node (5) [below right of =3] {5};
  \node (6) [below left of =3] {6};
  \node (7) [right of = 1] {7};
  \node (8) [right of = 7] {8};

\path[every node/.style={font=\sffamily\small}]
    (1) edge[ultra thick] node {} (2)
    (2) edge[ultra thick] node {} (3)
    (3) edge[ultra thick] node {} (1);

\begin{scope}[xshift=3cm]

  \node (1) {1};
  \node (2)  [below right of=1] {2};
  \node (3) [below of =2] {3};
  \node (4) [below of =3] {4};
  \node (5) [below right of =3] {5};
  \node (6) [below left of =3] {6};
  \node (7) [right of = 1] {7};
  \node (8) [right of = 7] {8};
  
\path[every node/.style={font=\sffamily\small}]
    (1) edge node {} (2)
    (2) edge node {} (3)
    (3) edge node {} (1)
    (3) edge[ultra thick] node {} (4)
    (3) edge[ultra thick] node {} (5)
    (5) edge[ultra thick] node {} (4);

\end{scope}

\begin{scope}[xshift=6cm]

  \node (1) {1};
  \node (2)  [below right of=1] {2};
  \node (3) [below of =2] {3};
  \node (4) [below of =3] {4};
  \node (5) [below right of =3] {5};
  \node (6) [below left of =3] {6};
  \node (7) [right of = 1] {7};
  \node (8) [right of = 7] {8};

\path[every node/.style={font=\sffamily\small}]
    (1) edge node {} (2)
    (2) edge node {} (3)
    (3) edge node {} (1)
    (3) edge[ultra thick] node {} (4)
    (4) edge[ultra thick] node {} (6)
    (6) edge[ultra thick] node {} (3)
    (3) edge node {} (5)
    (5) edge node {} (4);

\end{scope}

\begin{scope}[xshift=9cm]

  \node (1) {1};
  \node (2)  [below right of=1] {2};
  \node (3) [below of =2] {3};
  \node (4) [below of =3] {4};
  \node (5) [below right of =3] {5};
  \node (6) [below left of =3] {6};
  \node (7) [right of = 1] {7};
  \node (8) [right of = 7] {8};

\path[every node/.style={font=\sffamily\small}]
    (1) edge node {} (2)
    (2) edge node {} (3)
    (3) edge node {} (1)
    (3) edge node {} (4)
    (4) edge[ultra thick] node {} (6)
    (6) edge node {} (3)
    (3) edge node {} (5)
    (5) edge[ultra thick] node {} (4)
    (6) edge[ultra thick] node {} (5);

\end{scope}

\begin{scope}[xshift=12cm]

  \node (1) {1};
  \node (2)  [below right of=1] {2};
  \node (3) [below of =2] {3};
  \node (4) [below of =3] {4};
  \node (5) [below right of =3] {5};
  \node (6) [below left of =3] {6};
  \node (7) [right of = 1] {7};
  \node (8) [right of = 7] {8};
  
\path[every node/.style={font=\sffamily\small}]
    (1) edge node {} (2)
    (2) edge node {} (3)
    (3) edge node {} (1)
    (3) edge node {} (4)
    (4) edge node {} (6)
    (6) edge node {} (3)
    (7) edge[ultra thick] node {} (8)
    (8) edge[ultra thick] node {} (2)
    (2) edge[ultra thick] node {} (7)
    (3) edge node {} (5)
    (5) edge node {} (4)
    (6) edge node {} (5);

\end{scope}

\end{tikzpicture}
\caption{Example of a triangle-induced graph, $\TIG,$  with 8 vertices spanned by 5 triangles. Here the graph is sequentially constructed as a union of triangles, and at each step the newly added triangle is highlighted using bold edges. Note that this graph is also spanned by 6 triangles.} \label{fig:triangle-induced}
\end{figure}

Note that, e.g., although a clique of size 4 contains four triangles, it can be spanned by three triangles as well.

In the remainder of this section we prove three main results. 
{First, we prove Lemma \ref{lemma 3.1} which states that $\TIG$s are far from trees, namely, it lower bounds the number of tree-excess edges in a $\TIG$ in terms of the total number of edges. Then, we prove Lemma \ref{lemma 3.2}, which upper bounds the number of triangles {in a connected graph} in terms of the number of tree-excess edges. The next result we prove is Lemma \ref{lemma 4.3} which states that any connected graph having not too many tree-excess edges contains a dense sub-graph which accounts for most triangles. We finish with the proof of a structural result, Proposition \ref{equality holder}, which states that a graph where the Kruskal-Katona bound is almost sharp, necessarily  contains a subgraph that is almost a clique, in the sense of the relation between the number of vertices, edges and triangles.
}

\begin{lemma}\label{lemma 3.1}
Suppose that $G = (V,E)$ is a $\TIG$. Then,
\begin{align*}
|E| \leq 5(|E|-|V| ) + 5 .
\end{align*}  
\end{lemma}

\begin{proof}
Let $T$ be a spanning tree of $G$ (recall that $G$ is connected), and pick an arbitrary vertex as the root and denote it by $\rho$.  We construct a map $\phi : E(T) \rightarrow E \setminus E(T)$  which is an at most 4 to 1 mapping. From now on, we call any element in $E \setminus E(T)$ an excess edge.

For any edge $e \in E(T)$, let $\Delta(e)$ be a triangle containing $e$; there may be several such in which case we choose an arbitrary one {(such a triangle always exists since $G$ is a $\TIG$)}. Note that since $T$ is a tree,  not all edges in  $\Delta(e)$ can belong to $E(T)$. This leads us to consider the following two cases.

\textbf{Type 1.} $\Delta(e)$ contains only one excess edge $E \setminus E(T)$:
We denote this excess edge by $\phi(e)$. 

\textbf{Type 2.} $\Delta(e)$ contains two   excess edges $E \setminus E(T)$:
{Let $v$ and $w$ be the  endpoints of the edge $e$. Since $d(\rho,v)\neq d(\rho,w)$ (recall that $e$ is an edge in a tree $T$),  without loss of generality, we  assume that $d(\rho,v)> d(\rho,w)$. Then, define $\phi(e)$ to be the excess edge in $\Delta(e)$ which has $v$ as an endpoint.}

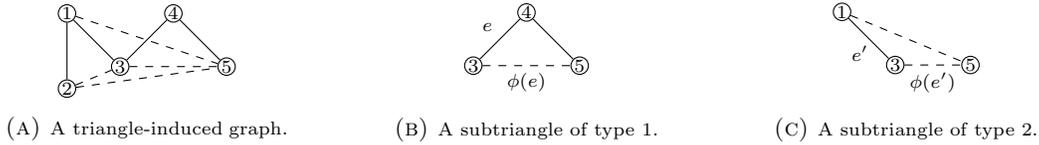
\begin{figure}[H]
\begin{tiny}
\begin{subfigure}{0.3\textwidth}
\centering
\begin{tikzpicture}[node distance=1cm, font = \fontsize{1}{0.1}, scale = 1]
\tikzstyle{every node}=[draw, shape=circle, minimum size = 0.1pt, inner sep=0.2]
  \node (1) {1};
  \node (2)  [below of = 1] {2};
  \node (3) [below right of = 1] {3};
  \node (4) [above right of = 3] {4};
  \node (5) [below right of = 4] {5};

\path[every node/.style={font=\sffamily\small}]
    (1) edge (2)
    (2) edge[dashed] (3)
    (3) edge (1)
    (3) edge (4)
    (4) edge (5)
    (2) edge[dashed] (5)
    (1) edge[dashed] (5)
    (3) edge[dashed] (5);
\end{tikzpicture}
\caption{\tiny{A triangle-induced graph.}}
\end{subfigure}
\begin{subfigure}{0.3\textwidth}
\centering
\begin{tikzpicture}[node distance=1cm, font = \fontsize{1}{0.1}, scale = 1]
\tikzstyle{every node}=[draw, shape=circle, minimum size = 0.1pt, inner sep=0.2]
  \node (4) {4};
  \node (3) [below left of = 4] {3};
  \node (5) [below right of = 4] {5};

\path[every node/.style={font=\sffamily\tiny}]
    (3) edge node[above left] {$e$}  (4)
   (4) edge (5)
    (3) edge[dashed] node[below] {$\phi(e)$} (5);
\end{tikzpicture}
\caption{\tiny{A subtriangle of type 1.}}
\end{subfigure}
\begin{subfigure}{0.3\textwidth}
\centering
\begin{tikzpicture}[node distance=1cm, font = \fontsize{1}{0.1}, scale = 1]
\tikzstyle{every node}=[draw, shape=circle, minimum size = 0.1pt, inner sep=0.2]
  \node (1) {1};
  \node (3) [below right of = 1] {3};
  \node (5) [right of = 3] {5};

\path[every node/.style={font=\sffamily\tiny}]
    (3) edge node[below left, xshift = 0.1cm] {$e'$}  (1)
    (1) edge[dashed] (5)
    (3) edge[dashed] node[below] {$\phi(e')$} (5);

\end{tikzpicture}
\caption{\tiny{A subtriangle of type 2.}}
\end{subfigure}
\end{tiny}
\caption{The solid edges in graph (A) form a spanning tree and the excess edges are dashed. We take node $1$ as the root of the spanning tree. For any edge of the spanning tree, the node with the higher label is further away from the root.} \label{fig:excess_edges2}
\end{figure}

\noindent
\textbf{Claim.} Each excess edge in $E \setminus E(T)$ is mapped to by at most four edges $e\in E(T)$.

To see this first note that any excess edge $f$ is mapped to by at most two edges of type 1. In fact, if an excess edge $f$ is contained in two triangles $\Delta_1$ and $\Delta_2$ that contain only one excess edge, then there is a cycle of size 4 in $T$, which cannot be true since $T$ is a tree.

Therefore, it suffices to show that any excess edge $f = {(v,w)}$ is mapped to by at most two edges $e\in E(T)$ of type 2.  Suppose that  there exist  two edges $e_1 = {(v_1,v)}$ and $e_2 = (v_2,v)$ in $  E(T)$ such that $\phi(e_1)=\phi(e_2) = f$. {Then, by the definition of $\phi$,
$d(\rho,v_1) = d(\rho,v_2)  = d(\rho,v)-1$}, which again violates the fact that $T$ is a tree. Hence, for any endpoint of an excess edge $f$, at most one edge $e\in E(T)$ of type 2 {incident on it}  which is mapped to $f$.  Thus, any  $f = (v,w)$ is mapped to by at most two edges  $e\in E(T)$ of type 2, concluding the proof of the claim.
 
 \noindent
Therefore, since $\phi : E(T) \rightarrow E \setminus E(T)$ is an at most 4 to 1 mapping, we have
\begin{align*}
  E(T)\leq 4( |E| - |E(T)|  ),
\end{align*}
which finishes the proof.

\end{proof}

{\begin{lemma}  \label{lemma 3.2}
Let $G = (V,E)$ be a connected graph. Then,
\begin{align*}
\Delta \leq \frac{\sqrt{2}}{3} (|E|-|V|+1)^{3/2} +3( |E| - |V| + 1).
\end{align*} 

\end{lemma}
}

\begin{proof}
As in the previous lemma,  let $T$ be a spanning tree of $G$ and $\rho$ be an arbitrary root. 
Since $T$ is a tree, each triangle $\Delta$ in $G$ is of one of the following three types.

Type 1. $\Delta$ has one excess edge.

Type 2. $\Delta$ has two  excess edges.

Type 3. $\Delta$ has three excess edges.

We define a map $\phi$ from  the set of  triangles  of type 1 or type 2 to $E \setminus E(T)$, the set of excess edges. For any triangle $\Delta$ of type 1, define $\phi(\Delta)$ to be the unique edge in $E \setminus E(T)$ contained in $\Delta$.  For any triangle $\Delta$ of type 2, let $e = (v,w)$ be the unique edge of $\Delta$ in $E(T)$.  Without loss of generality, we assume $d(\rho,v) > d(\rho,w)$. Then, define   $\phi(\Delta)$ to be {the edge} in $E \setminus E(T)$ which has $v$ as an endpoint.

{Since $T$ is a tree,  each edge in $E \setminus E(T)$ is mapped to by at most one triangle of type 1.} In addition,
by the same reasoning as in the proof of Lemma \ref{lemma 3.1}, each edge in $E \setminus E(T)$ is mapped to by at most two triangles of type 2. Hence,
\begin{align} \label{221}
|\text{type 1 triangles}| +|\text{type 2 triangles}| \leq  3(|E|-|E(T)|)   =  3(|E|-|V|+1).
\end{align}
In order to control the number of triangles of type 3, we use Lemma \ref{lemmaholder} upper bounding the number of triangles in a graph in terms of the number of edges,
\begin{align} \label{222}
|\text{type 3 triangles}| \leq  \frac{\sqrt{2}}{3} (|E|-|V|+1)^{3/2}.
\end{align}
Therefore, combining \eqref{221} and \eqref{222},   we conclude the proof.
\end{proof}

The next result states that if a connected graph has not too many tree excess edges, then it admits a dense subgraph, {in the sense that the number of edges is of the order of the square of the vertices}, which accounts for most of its triangles. 

{
\begin{lemma}  \label{lemma 4.3}
Let $\xi>0$ be a  sufficiently small constant.
Let  $G=(V,E)$ be a {connected} graph containing at least $\ell$  triangles such that 
\begin{align*} 
|E|-|V|  \leq  \frac{1}{2}\xi^{-1/2}  \ell^{2/3} - 1.
\end{align*} 
Then,    for  sufficiently large $\ell
$, there exists  a subgraph $G'$ in $G$ such that
\begin{align}\label{334}
 |V(G')| \leq  \xi^{-3/2} \ell^{1/3},
\end{align}
\begin{align}\label{333} 
|E(G')| \geq  \Big (\frac{3}{\sqrt{2}}\Big)^{3/2} (1- 2\xi^{1/2})^{2/3} \ell^{2/3},
\end{align}
\begin{align}\label{335}  
|\Delta(G')|  \geq (1-2 \xi^{1/2}) \ell.
\end{align}

\end{lemma} 
}
We need the following easy lemma.
\begin{lemma} \label{lemma 4.2}
Let 
$G = (V,E)$ be a {connected  graph}.
Then there is a  subgraph $G'$ such that
\begin{align} \label{420}
\Delta(G') \geq  \Delta(G) -  3(|E|-|V|+1) ,\qquad |E(G')| = |E| - |V| +1 .
\end{align} 
\end{lemma}

\begin{proof} Consider a spanning tree $T$ of $G$. 
Let  $G'$ be {the subgraph} induced by the excess edges $E \setminus E(T)$. {Then, in terms of the terminology in the proof of  Lemma \ref{lemma 3.2}, $\Delta(G') $ is nothing other than the number of type 3 triangles.  Therefore, by \eqref{221}, we have \eqref{420}.}
\end{proof}

Equipped with this lemma, we proceed to proving Lemma \ref{lemma 4.3}. 

\begin{proof}[Proof of Lemma \ref{lemma 4.3}]

 By Lemma \ref{lemma 4.2},  there exists a  subgraph $G_1$  such that 
\begin{align*}
\Delta(G_1) \geq \ell -    \frac{3}{2}\xi^{-1/2}  \ell^{2/3}   ,\qquad |E(G_1)| \leq   \frac{1}{2}\xi^{-1/2}  \ell^{2/3} .
\end{align*}
For $v\in V(G_1)$, let $d_v$ be {the} degree of $v$ in $G_1$.
Define 
\begin{align*}
A:= \{v \in V(G_1): d_v \leq \xi \ell^{1/3}\}.
\end{align*}
Then, the number of triangles in $G_1$ having at least one vertex in $A$ is bounded by
\begin{align} \label{331}
{\sum_{v\in A} d_v^2} \leq  \xi \ell^{1/3} \sum_{v\in A} d_v\leq   2  \xi \ell^{1/3}   |E(G_1)| \leq  \xi^{1/2}\ell.
\end{align}
Denote by $G'$ the subgraph of $G_1$ induced by the vertex set $V(G_1)\setminus A$. Then, for large enough $\ell$,
\begin{align*}
\Delta(G') \geq  \ell -   \frac{3}{2}\xi^{-1/2}  \ell^{2/3}    -{\ \xi^{1/2}\ell}  \geq (1-2 \xi^{1/2}) \ell.
\end{align*}  Thus,   we have \eqref{335}, and by the bound in Lemma \ref{lemmaholder},
we subsequently obtain \eqref{333}. In addition, using the fact 
\begin{align*}
\xi^{-1/2}\ell^{2/3} \geq 2 E(G_1)  \geq    \sum_{v\in G'} d_v \geq |V(G')| \xi \ell^{1/3},
\end{align*}
we have \eqref{334}.
\end{proof}

The final result of this section provides a structural description of when \eqref{holder} is almost sharp. As has already been mentioned in Section \ref{iop}, after the completion of the paper, Noga Alon pointed out to us that a similar ``stability" result for the Kruskal-Katona theorem had been proved earlier in \cite{keevash08}. However, since our proof appears to be different and could be of independent interest, we decided to include the proof. 
 \begin{proposition} \label{equality holder}
 Let $\e>0$ be a sufficiently small constant.
Let $G=(V,E)$ be a graph with at least $ \frac{1}{6} (1-\varepsilon) k$ many triangles and at most $ (1+\e) \frac{1}{2} k^{2/3}$ many edges. Then, there is a subset  $V' \subseteq V$ such that $|V'| \leq (1+6\e^{1/4})k^{1/3} $ and $\Delta(V') \geq \frac{1}{6}(1-6\e^{1/4})k$.
\end{proposition}

In order to prove this, we analyze the principal eigenvector of the adjacency matrix of $G$.  Note that the top eigenvector of a clique is uniformly distributed. We show that the top eigenvector of $G$ is `almost' uniformly {distributed on some subset $V' \subseteq V$ with $|V'| \approx k^{1/3}$. }
\begin{proof}
Let $\lambda_1\geq  \cdots \geq \lambda_n$ and $v_1,\cdots,v_n$ be eigenvalues and eigenvectors of the adjacency matrix of $G$ which is denoted by $A$. Then,
{\begin{align*}
\text{tr}(A^2) = \sum \lambda_i^2 &= 2  |E| \leq   (1+\e)   k^{2/3}  , \\
\text{tr}(A^3) = \sum \lambda_i^3   &=6\Delta   \geq  (1-\varepsilon ) k.
\end{align*}}
{Since
$
\sum \lambda_i^3 \leq  \lambda_1  \sum \lambda_i^2 
$ (recall that $\lambda_1$ is positive by the Perron-Frobenius theorem)}, we have
\begin{align} \label{412}
 \frac{1-\e}{1+\e} k^{1/3} \leq  \lambda_1 \leq  (1+\e)   ^{1/2}k^{1/3}  ,
\end{align}
which implies
\begin{align} \label{410}
 \lambda_2^2 + \cdots + \lambda_n^2 \leq  \Big(1+ \e -\Big (\frac{1-\e}{1+\e}\Big)^2\Big)  k^{2/3}.
\end{align}

{By \eqref{412} and \eqref{410}, $\lambda_2,\cdots,\lambda_n$ are negligible compared to $\lambda_1$ for small enough $\e>0$.  Motivated by this, we decompose}
\begin{align*}
A = \lambda_1 v_1 v_1^T + \sum_{i=2}^n \lambda_i v_i v_i^T =: A_1 + A_2.
\end{align*}
Define
\begin{align} \label{425}
\tilde{\e}:= 1+ \e - \Big(\frac{1-\e}{1+\e}\Big)^2.
\end{align} 
Then, for small enough $\e>0$,
{
\begin{align} \label{433}
\e < \tilde{\e} < 8\e,
\end{align}
}
and by \eqref{410}, 
\begin{align} \label{413}
\norm{A_2}_F^2 &= \sum_{i=2}^n \lambda_i^2  \leq  \tilde{\e} k^{2/3}.
\end{align}
Letting $v_1 = ( (v_1)_i)_{i=1,\cdots,n}$, we next  partition the {vertex set $V = \{1,\cdots,n\}$} in the following way.
\begin{align*}
S_1&:= \Big\{ i:  {| (v_1)_i  |} > \frac{1+ \tilde{\e}^{1/4} }{k^{1/6}} \Big \} , \\
S_2&:=\Big \{ i:   \frac{\tilde{\e}^{1/8}}{k^{1/6}}   <  | (v_1)_i  | < \frac{1-\tilde{\e}^{1/4}}{k^{1/6}}\Big \}  ,\\
S_3 &:=\Big \{ i:   | (v_1)_i  | \leq  \frac{ \tilde{\e}^{1/8} }{k^{1/6}} \Big\},  \\
S_4 &:= \Big\{ i:  \frac{1- \tilde{\e}^{1/4} }{k^{1/6}} \leq  | (v_1)_i  | \leq  \frac{1+ \tilde{\e}^{1/4} }{k^{1/6}} \Big\}.
\end{align*}
{We will ultimately take $S_4$ to be $V'$ appearing in the statement of the proposition. To be able to show this, we need to establish the contributions coming from the remainder of the graph to be negligible. 
}\\

\noindent
Towards this we first show that 
\begin{align} \label{415}
|S_1|  < 2\tilde{\e}^{1/4}   k^{1/3}.
\end{align}
Using \eqref{412} and the definition of $S_1,$ it follows that each entry of the sub-matrix obtained by the restriction of $A_1=  \lambda_1 v_1 v_1^T$ to $S_1\times S_1$ has absolute value greater than $(1+\tilde{\e}^{1/4})^2 \frac{1-\e}{1+\e} $. Since each entry of $A$ is either 0 or 1 (as it is the adjacency matrix of a graph), 
{for small $\e>0$, absolute value of  each entry in the sub-matrix of $A_2 =  A-A_1$ restricted to $S_1\times S_1$ is greater than
\begin{align*}
  (1+\tilde{\e}^{1/4})^2 \frac{1-\e}{1+\e} -1   > (1+2\tilde{\e}^{1/4}) \frac{1-\e}{1+\e} -1  =  2\tilde{\e}^{1/4} - \frac{2\e}{1+\e}  (1+2\tilde{\e}^{1/4}) > 2\tilde{\e}^{1/4}   - 4\e \overset{\eqref{433}}{>}   \frac{1}{2} \tilde{\e}^{1/4} .
\end{align*}
This implies that the Frobenius norm of  $A_2 $ restricted to $S_1\times S_1$ is greater than
$
 \frac{1}{2}\tilde{\e}^{1/4}   \sqrt{2} |S_1|.
$
Hence,  by \eqref{413}, 
we obtain  \eqref{415}.
}

\noindent
{Next, we show that
\begin{align} \label{416}
|S_2| < 2 \tilde{\e}^{1/4}   k^{1/3}.
\end{align}
Using \eqref{412} and the definition of $S_2,$ it follows that  each entry of the sub-matrix $A_1=   \lambda_1 v_1 v_1^T$ restricted to $S_2\times S_2$ has an absolute value between $\tilde{\e}^{1/4} \frac{1-\e}{1+\e} $ and $(1-\tilde{\e}^{1/4})^2 (1+\e)^{1/2}$. For small enough $\e>0$,  $\tilde{\e}^{1/4} \frac{1-\e}{1+\e}  > \frac{1}{2}\tilde{\e}^{1/4} $ and 
 \begin{align*}
1-(1-\tilde{\e}^{1/4})^2 (1+\e)^{1/2} > 1-(1-\tilde{\e}^{1/4}) \Big(1 + \frac{1}{2}\e\Big) > \tilde{\e}^{1/4} - \frac{1}{2}\e \overset{\eqref{433}}{>}  \frac{1}{2}\tilde{\e}^{1/4}.
\end{align*}
Since each entry of $A$ is either 0 or 1, each entry of $A_2 $ restricted to $S_2\times S_2$ is greater than  $\frac{1}{2}\tilde{\e}^{1/4}$. This implies that the Frobenius norm of  $A_2 $ restricted to $S_2\times S_2$ is greater than
$
 \frac{1}{2}\tilde{\e}^{1/4}   \sqrt{2} |S_2|,
$
and in conjunction with \eqref{413} implies \eqref{416}.

\noindent
In addition, since $\norm{v}_2=1$, by Markov's inequality}
\begin{align} \label{417}
|S_4| \leq   \frac{1}{(1-\tilde{\e}^{1/4})^2} k^{1/3}.
\end{align}

{Recalling that $\Delta(W)$ denotes the number of triangles consisting of vertices  in $W$ and that 
our goal is to prove that most of the triangles in $G$ come from $S_4$, we next show that for small enough $\e$,}
\begin{align} \label{414}
\Delta(G) - \Delta(S_4) \leq 3 \tilde{\e}^{1/4}   k.
\end{align}

The remainder of the proof is devoted to proving the above. There are quite a few technical steps.  {First, using the bounds from \eqref{415} and \eqref{416}, we bound   the number of triangles having at least one vertex in $S_1$ or $S_2$. Then, using the fact that $|(v_1)_i|$ is small for $i\in S_3$, we argue that there cannot be many triangles having at least one vertex in $S_3$. }

We start by defining some new notation. {For any $W_1,W_2\subseteq V$, let $\Delta_1(W_1,W_2)$ be the set of triangles having one vertex in $W_1$ and two vertices in $W_2$,  $\Delta_2(W_1,W_2)$ be the set of triangles having two vertices in $W_1$ and one vertex in $W_2$, and set $\Delta(W_1,W_2) = \Delta_1(W_1,W_2)+\Delta_2(W_1,W_2)$. In addition,  for   $W_1,W_2,W_3\subseteq V$,   let $\Delta(W_1,W_2,W_3)$ be the set of triangles with one vertex each in $W_1,W_2,W_3$ respectively.\\

\noindent
Since $\{S_1,S_2,S_3,S_4\}$ is a partition of the vertex set,}
\begin{align} \label{411}
\Delta(G) - \Delta(S_4)& = \Delta (S_1) +   \Delta (S_2) +  \Delta (S_3) +  \Delta (S_1,S_2) +  \Delta (S_1,S_3)+ \Delta (S_1,S_4)\nonumber  \\
&+ \Delta (S_2,S_3)+ \Delta (S_2,S_4)+ \Delta (S_3,S_4) \nonumber \\
&+ \Delta(S_1,S_2,S_3)+  \Delta(S_1,S_2,S_4)+  \Delta(S_1,S_3,S_4)+  \Delta(S_2,S_3,S_4).
\end{align}
We now seek to control each term.\\

\noindent
{
By \eqref{415},   using $|E| < (1+\e) \frac{1}{2}k^{2/3}$,  
\begin{align} \label{419}
 \Delta (S_1) + \Delta (S_1,S_2)+  \Delta (S_1,S_3) + \Delta (S_1,S_4)  &+  \Delta(S_1,S_2,S_3) + \Delta(S_1,S_3,S_4) \nonumber \\
 &\leq |S_1| \cdot |E| \leq  \tilde{\e}^{1/4} (1+\e) k.
\end{align}
The above uses the trivial bound that the number of triangles with one endpoint in $S_1$ is at most {$|S_1|\cdot |E|.$}\\

\noindent
Similarly, by
\eqref{416},
\begin{align} \label{4190}
  \Delta (S_2) +  \Delta (S_2,S_1) + \Delta(S_2,S_3) + \Delta(S_2,S_4) &+  \Delta(S_1,S_2,S_4)+ \Delta(S_2,S_3,S_4) \nonumber \\
  &  \leq |S_2| \cdot |E|    \leq \tilde{\e}^{1/4} (1+\e) k.
\end{align}
}

Consider the subgraph $G_1$ induced by edges with  both endpoints in $S_3$, or  one in $S_3$ and $S_2\cup S_4$ each. We show that
\begin{align} \label{418}
|E(G_1)| <  \tilde{\e}  k^{2/3}.
\end{align} 
{
Note that by definition of $S_2,S_3,S_4$, for any vertices $i,j$ connected by an edge in $G_1$,
\begin{align*}
|(v_1)_i (v_1)_j| \leq  \frac{\tilde{\e}^{1/8}}{k^{1/6}}  \cdot  \frac{1+\tilde{\e}^{1/4}}{k^{1/6}} .
\end{align*}
 By \eqref{412}, the absolute value of each entry of $A_1$ restricted to $G_1$ is (upper) bounded by
\begin{align*}
  \frac{\tilde{\e}^{1/8}}{k^{1/6}}  \cdot  \frac{1+\tilde{\e}^{1/4}}{k^{1/6}} \cdot (1+\e)^{1/2} k^{1/3} = (1+\e)^{1/2}  \tilde{\e}^{1/8}  (1+ \tilde{\e}^{1/4}).
\end{align*} 
Suppose that \eqref{418} fails, i.e. that there are at least $ \tilde{\e} k^{2/3}$  many edges in $G_1$.} Then, in the matrix $A_2$, there are at least  $2  \tilde{\e}  k^{2/3}$ many entries greater than $1- (1+\e)^{1/2}  \tilde{\e}^{1/8}  (1+ \tilde{\e}^{1/4}) $. This implies that if \eqref{418} fails, then the Frobenius norm of $A_2$ is greater than
\begin{align*}
\sqrt{2} \tilde{\e}^{1/2} k^{1/3} ( 1- (1+\e)^{1/2}  \tilde{\e}^{1/8}  (1+ \tilde{\e}^{1/4})  ) .
\end{align*} 
This contradicts  \eqref{413} for small $\e>0$. Hence, we have \eqref{418} which in particular implies that the number of edges in the subgraph induced by $S_3$, denoted by  $|E(S_3)|$, satisfies
 \begin{align} \label{441}
 |E(S_3)| \leq |E(G_1)| \leq \tilde{\e}k^{2/3}.
\end{align}   Hence, by Lemma \ref{lemmaholder},
\begin{align} \label{421}
\Delta(S_3) \leq  \frac{\sqrt{2}}{3} \tilde{\e}^{3/2} k.
\end{align} 

Finally, in order to bound  \eqref{411}, the last quantity we need to bound is $ \Delta(S_3,S_4)$. We  show that
\begin{align} \label{424}
 \Delta(S_3,S_4) \leq \Big ( \frac{\tilde{\e}}{(1-\tilde{\e}^{1/4})^2}  +   (1+\e)  \tilde{\e}^{1/2}  \Big) k .
\end{align} 
First,  by \eqref{417} and \eqref{441},
\begin{align} \label{422}
\Delta_2(S_3,S_4)  \leq |E(S_3)| \cdot |S_4|  \leq    \frac{\tilde{\e}}{(1-\tilde{\e}^{1/4})^2}  k.
\end{align}

Next,
in order to bound $|\Delta_1(S_3,S_4)|$,
let $d_1, \cdots, d_m$ be the number of neighbors in $S_4$ of
each vertex in $S_3$, i.e., the degree of each vertex of $S_3$ in $S_4$, and let \begin{align*}
W := \{i\in S_3: d_i < \tilde{\e}^{1/2} k^{1/3} \}.
\end{align*} 
By \eqref{418},
\begin{align*}
|W^c\cap S_3| < \tilde{\e}^{1/2} k^{1/3}.
\end{align*}
Since the number of edges in the subgraph induced by $S_4$ is { upper bounded by the total number of edges} $ \frac{1}{2}(1+\e) k^{2/3}$,
the number of triangles  having one vertex in $W^c\cap  S_3$ and two vertices in $S_4$ is bounded by $ \frac{1}{2}(1+\e) \tilde{\e}^{1/2} k$. In addition, using \eqref{418}, the number of triangles  having one vertex in $W$ and two vertices in $S_4$ is bounded by
\begin{align*}
 \sum_{i\in W}d_i^2  \leq \big (\max_{i\in W} d_i \big) (d_1+\cdots+d_m) \leq  \tilde{\e}^{1/2} k^{1/3} \cdot  \tilde{\e} k^{2/3} = \tilde{\e}^{3/2} k.
\end{align*} 
Thus, we have
\begin{align} \label{423}
\Delta_1(S_3,S_4) \leq  { \frac{1}{2}(1+\e) \tilde{\e}^{1/2} k +  \tilde{\e}^{3/2} k \leq}  (1+\e)  \tilde{\e}^{1/2} k.
\end{align}
Combining \eqref{422} and \eqref{423}, we have \eqref{424}. {Therefore, applying \eqref{419}, \eqref{4190} and \eqref{421} to \eqref{411}, we obtain \eqref{414}. 
 Hence,  
 \begin{align*}
  \Delta(S_4) \geq \Delta(G) - 3 \tilde{\e}^{1/4}k \geq  \Big ( \frac{1}{6}(1-\e) - 3 \tilde{\e}^{1/4} \Big ) k  \overset{\eqref{433}}{\geq}   \frac{1}{6}(1- 6\e^{1/4})k.
\end{align*}  In addition, by \eqref{417},
\begin{align*}
|S_4| \leq  \frac{1}{(1-\tilde{\e}^{1/4})^2}  k^{1/3}\leq  (1+3\tilde{\e}^{1/4})  k^{1/3}  \overset{\eqref{433}}{ \leq}  (1+6\e^{1/4})k^{1/3}.
\end{align*} 
This implies that the subset $S_4 $ satisfies the desired property.
}
\end{proof}

\section{Rarity of triangle induced graphs}\label{convexsec}

We start by defining the function alluded to in Section \ref{iop}, and establish several of its properties. Using the latter we then obtain the estimate in \eqref{002}.
We start by defining a function $f:[0,\infty) \rightarrow (0,\infty)$ governed by the inequality in Lemma \ref{lemma 3.2}, by 
\begin{equation}\label{f1}
f(x) = \frac{\sqrt{2}}{3}(x+1)^{3/2}  + 3(x+1).
\end{equation}
$h$ will be essentially the inverse of $f$, which we now formally define.

Many of the forthcoming arguments establishing its properties are rather technical in nature, and can be safely skipped on first read.\\

Let $h: (10,\infty) \rightarrow (1,\infty)$ be the inverse function of $f$, restricted to $(10,\infty)$. {The number 10 is chosen so that $\inf_{y>10} h(y)>1$, which follows from the fact $f(1)<10$.}  Since $f$ is convex, $h$ is concave on $(10,\infty)$. {While our applications will only make use of the values of $h(y)$ for integers $y$, it would be convenient for expository reasons to regard $h$ as a function of real numbers which we define by setting: $h(y)=1$ for $2\leq y\leq 10$ and $h(y)=0$ for $y<2$. 

We next make some observations about the function $h,$ that will be of use in the sequel. {Notice that $h$ is monotone, and  since $h$ is concave on $(10,\infty),$ the mean value theorem implies that  for any $x,y\geq 10$,
\begin{align} \label{eq:h_subadd_10}
h(x)+h(y)  \geq  h(10 ) + h(x+y-10).
\end{align} }
Since $f(x) \geq \frac{\sqrt{2}}{3}x^{3/2},$ and recalling $h(y)\leq 1$ for $y\leq 10$, setting $a := (\frac{3}{\sqrt{2}})^{2/3}$, for all $y\geq 1$,
\begin{align} \label{h(y)}
 h(y) \leq ay^{2/3}.
\end{align}
Also, since $f(x) \geq \frac{\sqrt{2}}{3}(x+10)^{3/2}$ for large   $x$,  for large enough $y$,
\begin{align}\label{hy2}
h(y) + 10 \leq ay^{2/3}.
\end{align}
In addition, for any $\kappa>0$, for large enough $y>0$,
\begin{align} \label{hy}
h(y) > (a-\kappa)y^{2/3}.
\end{align}
Note that for $y>10$,
\begin{align} \label{derivative}
 h'(y) = f'(  h(y) )^{-1} =  \Big ( \frac{\sqrt{2}}{2}(h(y)+1)^{1/2} + 3 \Big )^{-1} .
\end{align} Thus, by \eqref{h(y)},  \eqref{hy} and the fact that $(\frac{\sqrt{2}}{2}a^{1/2})^{-1} = \frac{2}{3}a$ by our choice of $a$, for any $\kappa>0$, for large enough $y$,
\begin{align} \label{h'(y)}
(a-\kappa)  \frac{2}{3}y^{-1/3}\leq h'(y) \leq  (a+\kappa)  \frac{2}{3}y^{-1/3}.
\end{align}
Finally, let us examine the third derivative of $h$. By taking the derivative of $f(h(y))=y$ three times, we get
\begin{align*}
f'''(h(y))h'(y)^3 + 	3f''(h(y))h'(y)h''(y) + f'(h(y))h'''(y)=0.
\end{align*}
It is straightforward to verify that $f',f''>0,f'''<0 $. Also, since $h$ is increasing and concave on $(10,\infty)$, we have $ h''<0, h'>0$ on this interval. Thus, for   $y>10$,
\begin{align} \label{third}
h'''(y)>0.
\end{align}

We next record the following estimate. 
\begin{lemma} \label{lemma 3.0}
For any non-negative integers {$i$ and $m$},
\begin{align} \label{114}
i+ h(m-10i) \geq  h(m-10^5)  
\end{align} 
and
\begin{align} \label{1140}
h(m-10^5)   \geq    h(m) - h(10^6).
\end{align} 
\end{lemma}
The numbers $10^5$ and $10^6$ have no significant importance and they are chosen to be simply sufficiently large constants.

\begin{proof}
We first prove \eqref{114}. {When $m \leq 10^5$, the RHS of \eqref{114} becomes 0, and when $ 10i \leq 10^5$,  $h(m-10i)\geq   h(m-10^5)  $ by the monotonicity of $h$, so in those cases the result follows trivially. Thus it remains to prove the result for $m > 10^5$ and $10 i > 10^5$.}

We next divide the analysis into two cases:  $m\geq 10i+10^4$ and  $m< 10i+10^4$.
In the case $m\geq 10i+10^4$,
by the mean value theorem, 
\begin{align*}
 h(m-10^5)- h(m-10i) \leq  (10i -10^5)  h'(m-10i) \leq {10  i  h'(10^4)}\leq  i
\end{align*}   
{(recall that $h'$ is decreasing on $(10,\infty)$ since $h$ is concave on that interval). The last inequality follows since $f(600)<10^4$ and hence $h(10^4)>600$ which along with \eqref{derivative} implies $h'(10^4) < 0.1$.}
On the other hand, if $ m\leq  10i+10^4$, then  by \eqref{h(y)}, when $m \geq 10^5$,
\begin{align*}
i \geq \frac{1}{10}(m-10^4) \geq am^{2/3} \geq  h(m-10^5).
\end{align*}
To prove
\eqref{1140}, it suffices to consider the case $m\geq 10^6$ since otherwise the LHS is non-negative while the RHS is non-positive. By the mean value theorem,
\begin{align*}
h(m) - h(m-10^5)  \leq 10^5  h'(m-10^5) \leq 10^5  h'(10^6-10^5) \leq  h(10^6).
\end{align*}
The last inequality follows from \eqref{derivative} by observing that $h(10^6) , h(10^6-10^5)>14000$.
\end{proof}

\subsection{Estimates using $h$}
Now given the function $h,$ we have the following estimate on the number of tree excess edges in a graph in terms of the number of triangles.
\begin{lemma}For any {connected graph} $G = (V,E)$, if $\Delta \geq 1$,
\begin{align} \label{key}
|E| - | V| \geq  h(\Delta).
\end{align}
\end{lemma}
\begin{proof}
The proof is a consequence of Lemma \ref{lemma 3.2}.  In fact, recalling $h:(10,\infty) \rightarrow (1,\infty)$ is the inverse of $f$, \eqref{key} is true for $\Delta > 10$. Also, for $2\leq \Delta \leq 10$,  $|E|-|V|\geq 1 =  h(\Delta)$. This is because if $|E|-|V|=0$, then the number of triangles in $G$ is at most 1 (we obtain at most one triangle once  one additional edge is added to a spanning tree, {this is where the connectivity is crucially used}).  Finally, for $\Delta=1$,  $|E|-|V| \geq 0 = h(\Delta)$, since $G$ cannot be a tree.
\end{proof}

\begin{remark}\label{remark1}
Observe that
\eqref{key} implies that for any $\TIG = (V,E)$ spanned by $\ell\geq 1$  triangles,
\begin{align} \label{key1}
|E| - | V| \geq  h(\ell).
\end{align}
This is because $\Delta \geq \ell$ and $h$ is monotone.
\end{remark}

\noindent
We now come to the following key estimate. 
For a positive integer $\ell$, let $F_\ell$ be the event that there exists  a triangle-induced subgraph (to be denoted $\TISG$) spanned by $\ell$ triangles. In the next lemma, we bound the probability of $F_\ell$.

\begin{lemma} \label{lemma 3.3}
Let $\mu>0$ be a constant.
For any  constant $\eta>0$,  for sufficiently large $n$ and any  $ 1  \leq   \ell\leq  n^{\frac{1}{10}-\mu}$,
 \begin{align*}
 \mathbb{P}(F_\ell) \leq   \left(\frac{1}{n}\right)^{(1-\eta) h(\ell) }.
 \end{align*}
\end{lemma}

\begin{proof}
For positive integers $\ell,v,e$, 
let
$F_{\ell,v,e}$ be the event that there exists a $\TISG$ spanned by  $\ell$ triangles with $v$  vertices and  $e$ edges.
The number of {subgraphs with $v$ vertices and $e$ edges is bounded by $\binom{n}{v} \binom{{v \choose 2}}{e} \leq n^v v^{2e}$.}   Since   $e \leq 5(e-v)+5$ {by Lemma \ref{lemma 3.1}}, \begin{align} \label{332}
 \mathbb{P}(F_{\ell,v,e}) \leq \frac{d^e}{n^{e}} n^vv^{2e}= \frac{d^e}{n^{e-v}} v^{2e} \leq   \Big  ( \frac{d^5v^{10}}{n}\Big)^{e-v}  d^5v^{10}.
\end{align}

\noindent
By \eqref{key1}, we have
\begin{align*}
\mathbb{P}(F_\ell) \leq \mathbb{P} ( \cup_{e-v\geq h(\ell)}  F_{\ell,v,e}).
\end{align*}
It turns out that a naive union bound applied to the RHS using \eqref{332} does not yield a sharp bound, unless $\ell$ is small. For large $\ell$,  we use Lemma \ref{lemma 4.3} instead.
The estimates for large and small $\ell$, {where large and small are defined in terms of a parameter $\xi$ to be defined shortly}, are obtained separately in the following two steps respectively. \\

\noindent
\textbf{Step 1.}  Fixing a large constant  $b$ to be chosen later, we have 
\begin{align}\label{split}
\mathbb{P}(F_\ell) \leq \mathbb{P} ( \cup_{ e-v\leq  b\ell^{2/3}-1}  F_{\ell,v,e}) + \mathbb{P} ( \cup_{e-v\geq  b\ell^{2/3}-1}  F_{\ell,v,e}) .
\end{align}
In the second term, a naive union bound effectively works  since there are many tree-excess edges. In order to control the first term, we first apply Lemma \ref{lemma 4.3} which states that the event
$\cup_{e-v \leq  b  \ell^{2/3}-1} F_{\ell,v,e}$  implies the existence of a  dense subgraph accounting for most of the triangles. {More precisely, for a sufficiently small constant $\xi>0$, set $b=\frac{1}{2}  \xi^{-1/2}$. Then, 
 by Lemma \ref{lemma 4.3}, under the event $\cup_{e-v \leq  b  \ell^{2/3}-1} F_{\ell,v,e}$,} there exists  a subgraph $G'$ such that
\begin{align*}
 |V(G')| \leq  \xi^{-3/2} \ell^{1/3}, \qquad 
|E(G')| \geq a(1-2\xi^{1/2})^{2/3} \ell^{2/3}.
\end{align*}
{Hence, by a union bound and  the fact that the number of graphs with $v$ vertices  is bounded by $2^{v^2}$, for any small enough $\xi>0$, for large enough $\ell$ and $n\geq 2^{\xi^{-5}}$},
\begin{align} \label{341}
 \mathbb{P}(\cup_{e-v \leq  b l^{2/3}-1}  F_{\ell,v,e})& \leq   n^{ \xi^{-3/2}  \ell^{1/3}} 2^{  \xi^{-3}   \ell^{2/3}} \Big(\frac{d}{n}\Big)^{  a(1- 2 \xi^{1/2} )^{2/3} \ell^{2/3} } \nonumber  \\
 & \overset{\eqref{h(y)}}{\leq}  {  n^{ \xi^{-3/2}  \ell^{1/3}} 2^{  \xi^{-3}  \ell^{2/3}}  \Big(\frac{d}{n}\Big)^{(1- 2  \xi^{1/2} )^{2/3} h(\ell) } } \nonumber \\
 &\leq  2^{  \xi^{-3}  \ell^{2/3}}  \Big(\frac{d}{n}\Big)^{(1- 2  \xi^{1/2} )h(\ell) } \leq  \Big(\frac{d}{n}\Big)^{(1-3 \xi^{1/2})h(\ell) } ,
\end{align}
{where in the second last inequality, we use the fact that $h(\ell)$ is at least of order $\ell^{2/3}$ (see \eqref{hy}) to absorb the $n^{ \xi^{-3/2}  \ell^{1/3}}$ term by reducing the exponent of the $d/n$ term from ${(1- 2  \xi^{1/2} )^{2/3} h(\ell) }$ to ${(1- 2  \xi^{1/2} )h(\ell) }$. To see why the last inequality holds, first note that this is equivalent to}
$
 2^{  \xi^{-3}  \ell^{2/3}}  \leq    (\frac{n}{d})^{  \xi^{1/2}h(\ell) } .
$  Since $h(\ell) > \frac{a}{2} \ell^{2/3}$ for large $\ell$ by \eqref{hy}, it suffices to check $
 2^{  \xi^{-3.5} }  \leq    (\frac{n}{d})^{   a/2 } ,
$ { which holds for small $\xi>0$ since $n\geq 2^{\xi^{-5}}$ (any exponent greater than 3.5 works, the exponent 5 is an arbitrary choice).}\\
\noindent
To bound the other term in \eqref{split}, namely $\mathbb{P} ( \cup_{e-v\geq  b\ell^{2/3}-1}  F_{\ell,v,e})$  it suffices to bound
$$ \displaystyle{\sum_{e\geq v+b \ell^{2/3}-1}  \mathbb{P}(F_{\ell,v,e}).}$$
Note that the number of summands is less than $(3\ell)^2$, since the number of edges and vertices are bounded by $3\ell$. Thus,
using \eqref{332},  for any $\kappa>0$, for large enough $\ell$,
\begin{align} \label{343}
   \sum_{e\geq v +b \ell^{2/3}-1}  \mathbb{P}(F_{\ell,v,e})   
 &   \leq  (3\ell)^2   \Big( \frac{d^5(3\ell)^{10}}{n}\Big)^{ b \ell^{2/3}-1}  d^5(3\ell)^{10}   \nonumber   \\
 &  {\leq} C \ell^{( 10+\kappa)b\ell^{2/3}+12} \Big(\frac{1}{n}\Big)^{ b\ell^{2/3}-1} .
\end{align}  
 Since $b\rightarrow \infty$ as $\xi \rightarrow 0$ and $\ell \leq n^{\frac{1}{10} - \mu}$, this quantity is bounded by {$(\frac{1}{n})^{a\ell^{2/3} } \overset{\eqref{h(y)}}{ \leq  }(\frac{1}{n})^{h(\ell) }$} for small enough $\xi,\kappa>0$.
Therefore, applying \eqref{341} and \eqref{343} to \eqref{split}, for small enough $\xi>0$, there exists  $\ell_0(\xi)>0$ such that for    $\ell\geq \ell_0(\xi)$ and large enough $n$, 
\begin{align*}
 \mathbb{P}(F_\ell) \leq  \Big(\frac{1}{n}\Big)^{(1-4 \xi^{1/2})h(\ell) } .
\end{align*}
\textbf{Step 2.}  For small $\ell$, a simple union bound using \eqref{332} suffices.
For $1\leq \ell \leq \ell_0(\xi)$ and sufficiently large $n$ (depending on $\xi$), {there exists a constant $C = C(\xi)>0$ such that
\begin{align*}
 \mathbb{P}(F_\ell) \leq    \sum_{e\geq v+h(\ell)}  \mathbb{P}(F_{\ell,v,e})    & \leq  (3\ell)^2 \Big  ( \frac{d^5(3\ell)^{10}}{n}\Big)^{h(\ell)}  d^5(3\ell)^{10}  \\
 &\leq C\Big(\frac{1}{n}\Big)^{h(\ell)}  \leq \Big(\frac{1}{n}\Big)^{(1-\xi^{1/2})h(\ell)}.
\end{align*}
}
\end{proof}

Finally we have all the ingredients in place to prove Theorem \ref{theorem 1}.

\section{Proof of Theorem \ref{theorem 1} via graph decomposition into $\TISG$s.} \label{secupper}
Recall that due to the results of Section \ref{seclower}, it only remains to prove the corresponding upper bounds.

\noindent
For $1\leq \ell_1\leq \cdots \leq \ell_m$, let $E_{\ell_1,\cdots,\ell_m}$ be the event that there exist {vertex-disjoint} $\TISG$s  spanned by $\ell_i$  triangles.
We start by observing that the occurrence of the event $\{N\geq k\}$ implies  the occurrence of the event   $E_{\ell_1,\cdots,\ell_m}$ for some $m\geq 1$, $1\leq \ell_1\leq \cdots \leq \ell_m$ and $\ell_1+\cdots+\ell_m =  k$. In fact, if {there are  $N$  triangles}, then there are disjoint $\TISG$s  $H_1,\cdots,H_r$ in $G$ such that $H_i$ contains exactly $\ell_i'$  triangles with $1\leq \ell_1'\leq \cdots \leq \ell_r'$  and   $\ell_1'+\cdots+\ell_r'=N$. Since $N\geq k$, one can choose {$ \ell_i$ connected triangles in each $H_i$, such that $0 \leq \ell_i \leq \ell_i'$, $1\leq \ell_1\leq \cdots \leq \ell_m$, $m\leq r$,  and $\ell_1+\cdots+\ell_m=k$.} {Denoting by $H_1',\cdots,H_m'$  $\TISG$s spanned by each such collection of $\ell_1,\cdots,\ell_m$ triangles\footnote{We say that $H$ is the  $\TISG$ spanned by $\ell$ triangles $T_1,\cdots,T_\ell,$ if $H$ is connected and is the union of $T_1,\cdots,T_\ell$.}    respectively,  by considering the (vertex-disjoint) union of   these $\TISG$s,   one can deduce that  $\{N\geq k\}$ implies  the occurrence of $E_{\ell_1,\cdots,\ell_m}$.}

We now move towards upper bounding the probability $\mathbb{P} (E_{\ell_1,\cdots,\ell_m})$.
We assume that $1=\ell_1=\cdots =\ell_j < \ell_{j+1} \leq \cdots \leq \ell_m$.
By BK inequality,
\begin{align} \label{113}
\mathbb{P}(E_{\ell_1,\cdots,\ell_m}) \leq  \mathbb{P}(E_{1,\cdots,1})  \mathbb{P}(F_{\ell_{j+1}}) \cdots   \mathbb{P}(F_{\ell_m})
\end{align}
(there are $j$  1s in  $E_{1,\cdots,1}$).
We now obtain a rather precise estimate of $\mathbb{P}(E_{1,\cdots,1})$. The probability that there exist $j$  vertex-disjoint triangles is bounded by
\begin{align*}
\Big(\frac{d}{n}\Big)^{3j} \frac{1}{j!} {n \choose 3} {n-3 \choose 3} \cdots {n-3j+3 \choose 3} \leq  \frac{1}{j!} \Big(\frac{d^3}{6}\Big)^j.
\end{align*}
{Note that it is crucial that we do not  use the upper bound $\P(E_1)^j$ here, which would only give an upper bound $c^j$ ($c>0$ is a constant) on the probability  $\mathbb{P} ( E_{1,\dots, 1})$.}
This along with Lemma \ref{lemma 3.3} {(which is applicable since   $\ell_i \leq k  \leq n^{\frac{1}{10}-\mu}$ for $i=j+1,\cdots,m$)}, yield, for any $\eta>0$, for large enough $n$,
\begin{align} \label{E}
\mathbb{P}(E_{\ell_1,\cdots,\ell_m}) \leq  \mathbb{P}(E_{1,\cdots,1})  \mathbb{P}(F_{\ell_{j+1}}) \cdots   \mathbb{P}(F_{\ell_m}) \leq  \frac{1}{j!} \Big(\frac{d^3}{6}\Big)^j   \Big(\frac{1}{n}\Big) ^{ (1-\eta)(h(\ell_{j+1}) + \cdots + h(\ell_m)) }  .
\end{align}
Thus,
\begin{align}   \label{111}
 \mathbb{P}(N\geq k) \leq  \sum_{j=0}^k  \sum_{{\ell_{j+1} + \cdots + \ell_m  =  k-j}\atop {2\leq  \ell_{j+1}\leq \cdots\leq  \ell_m}}   \frac{1}{j!} \Big(\frac{d^3}{6}\Big)^j  \Big(\frac{1}{n}\Big) ^{(1-\eta) (h(\ell_{j+1}) + \cdots + h(\ell_m)) } .
\end{align}
{Recall from \eqref{eq:h_subadd_10}, that for any $x,y\geq 10$,}
\begin{align*}
h(x) + h(y) \geq h(10) + h(x+y-10) = 1+h(x+y-10).
\end{align*}
{We then apply this inequality repeatedly to $\ell_{j+1}, \cdots, \ell_m$  until there is at most one element greater than 10. Let us denote by $i$ the number of elements eventually less than or equal to 10. Since this procedure preserves the total sum to be equal to $k-j$, by the monotonicity of $h$, we obtain}
\begin{align*}
h(\ell_{j+1}) + \cdots + h(\ell_m)  \geq  i + h(k-j-10i).
\end{align*}
 Thus, applying Lemma \ref{lemma 3.0} and the fact that the number of partitions  of $n$ is bounded by $e^{\alpha\sqrt{n}}$ for some constant $\alpha>0$ (\cite{partition}), we get
{
 \begin{align} \label{110}
 \sum_{{\ell_{j+1} + \cdots + \ell_m  =  k-j}\atop {2\leq  \ell_{j+1}\leq \cdots\leq  \ell_m}}   \frac{1}{j!} \Big(\frac{d^3}{6}\Big)^j  \Big(\frac{1}{n}\Big) ^{(1-\eta) (h(\ell_{j+1}) + \cdots + h(\ell_m)) } \leq   \frac{1}{j!} \Big(\frac{d^3}{6}\Big)^j   e^{\alpha\sqrt{k-j}}    \Big(  \frac{1}{n}\Big)^{(1-\eta) h(k-j - 10^5)}  .
 \end{align}
 }
Hence,
the RHS of  \eqref{111} is  bounded by
\begin{align}  \label{112}
  \sum_{j=0}^{k}   \frac{1}{j!} \Big(\frac{d^3}{6}\Big)^j   e^{\alpha\sqrt{k-j}}    \Big(  \frac{1}{n}\Big)^{(1-\eta) h(k-j - 10^5)}  .
\end{align}
{
We now bound this in the cases  $k^{1/3} \log k< (a- \delta) \log n$ and  $k^{1/3} \log k> (a+ \delta) \log n$ where the $j \approx k$ and  $j \approx 0$ are the dominating terms respectively.} \\

\noindent
\textbf{Case 1.}
 $k^{1/3} \log k< (a- \delta) \log n$.

\noindent
We show that each term  in \eqref{112} is bounded by the term with $j=k$ up to a small multiplicative factor. Precisely,
we claim that there exists $a_j$ satisfying
\begin{align}  \label{116}
   \frac{1}{j!} \Big(\frac{d^3}{6}\Big)^j  e^{\alpha \sqrt{k-j}}   \Big(\frac{1}{n}\Big) ^ {(1-\eta)   h(k-j - 10^5) }  \leq    e^{a_j} \frac{1}{k!} \Big(\frac{d^3}{6}\Big)^k,
\end{align}
and further, for large enough $n$,
\begin{align} \label{117}
\sum_{j=0}^{k} e^{a_j} \leq Ce^{C\log k },
\end{align}
where $C = C(\delta,\eta)>0$ is a constant.  Clearly this finishes the argument, as, by \eqref{117}, the summation \eqref{112} is bounded by
\begin{align*}
Ce^{C\log k}  \frac{1}{k!} \Big(\frac{d^3}{6}\Big)^k  \leq  e^{-k\log  k + ck}
\end{align*}
for some constant $c>0$ depending only on $d$ (not depending on $\delta$) and large enough $k$.

Now to prove \eqref{116}, by taking a ratio of the LHS and RHS of the same,  recalling $d_0= \frac{d^3}{6}$ and  using $\frac{k!}{j!} \leq k^{k-j}$, we note that the quantity $a_j$ can be chosen as
\begin{align} \label{aj}
a_j :=  (k-j) \log (k/d_0)  + \alpha\sqrt{k-j} - (1-\eta)  h(k-j - 10^5)   \log n  .
\end{align}
Now, since by \eqref{h(y)} and  \eqref{hy}, $h(y)$ is asymptotically  $ay^{2/3}$, there exists  $M = M(\eta,\delta)>0$ such that for $k-j \geq  M$,
\begin{align*}
h(k-j - 10^5) \geq (1-\eta) h(k-j ) \geq  (1-\eta) (a-0.5\delta) (k-j)^{2/3},
\end{align*}
which implies
\begin{align*}
a_j &\leq (k-j) \log (k/d_0)  + \alpha\sqrt{k-j} - (1-\eta) ^2 (a-0.5\delta) (k-j)^{2/3}\log n\\
& \leq  (k-j)^{2/3} \left((k-j)^{1/3}  \log (k/d_0)+\alpha - (1-\eta)^2 (a- 0.5\delta) \log n \right).
\end{align*}
Since $k^{1/3}\log k < (a-\delta) \log n$,  for sufficiently small $\eta>0$, for $k-j \geq  M$,
\begin{align} \label{131}
a_j \leq  - (k-j)^{2/3}   \cdot 0.1    \delta  \log n  <0.
\end{align}

\noindent
Whereas, for $k-j\leq M$,  {by the expression of $a_j$ in \eqref{aj}},
\begin{align} \label{132}
a_j  \leq  M \log (k/d_0)  +   \alpha \sqrt{M}  .
\end{align}
Hence, by \eqref{131} and \eqref{132}, { 
 \begin{align*}
 \sum_{j=0}^{k} e^{a_j} =  \sum_{j=0}^{k-M} e^{a_j} + \sum_{j=k-M+1}^{k} e^{a_j}  \leq k +  Me^{ M \log (k/d_0)  +   \alpha \sqrt{M}  } \leq Ce^{C\log k},
\end{align*}   
which implies \eqref{117}.}

\noindent
\textbf{Case 2.}
$k^{1/3} \log k> (a+ \delta) \log n$.\\
By \eqref{1140}, the summation \eqref{112} is bounded by
\begin{align} \label{119}
  \sum_{j=0}^{k}   \frac{1}{j!} \Big(\frac{d^3}{6}\Big)^j   e^{\alpha\sqrt{k-j}}    \Big(  \frac{1}{n}\Big)^{(1-\eta) (h(k-j) - h(10^6))}. 
\end{align}
{
 Recalling $d_1= \frac{ed^3}{6}$ and  by Stirling's formula \eqref{stirling},  this is bounded by
 \begin{align} \label{125}
   \sum_{j=0}^{k}   \Big (\frac{d_1}{j}\Big)^j   e^{\alpha\sqrt{k}}    \Big(  \frac{1}{n}\Big)^{(1-\eta) (h(k-j) - h(10^6))}
 \end{align}
{(we set $(\frac{d_1}{0})^0 := 1$)}.
We claim that for small enough constant $\rho>0$, for sufficiently large $n$, each term above can be bounded by
\begin{align} \label{127}
C e^{\alpha\sqrt{k}} \Big(\frac{1}{n}\Big)^{(1-\eta)(h(k-\rho k) - h(10^6))}.
\end{align}
{
We first conclude the proof assuming this claim. Take  $\rho>0$ small enough so that for large $k$,
\begin{align*}
(1-\eta) (h(k-\rho k) - h(10^6)) > (1-2\eta)h(k).
\end{align*}   
This is possible since $h(y) \approx ay^{2/3}$ (see \eqref{h(y)} and \eqref{hy}). Hence,   \eqref{125} is bounded by
\begin{align*}
C k \cdot e^{\alpha \sqrt{k}}   \Big(\frac{1}{n}\Big)^{(1-\eta) (h(k-\rho k) - h(10^6)) } \leq C e^{\alpha \sqrt{k}} k    \Big(  \frac{1}{n}\Big)^{(1-2\eta)h(k) } \leq Ce^{h(k)}   \Big(  \frac{1}{n}\Big)^{(1-2\eta)h(k) } \leq    \Big(  \frac{1}{n}\Big)^{(1-3\eta)h(k) },
\end{align*}  
which concludes the proof.

\noindent
 Therefore, it suffices to  verify the claim, i.e.,  each term in \eqref{125} is bounded by \eqref{127}. This is equivalent to 
\begin{align}  \label{118}
 (1-\eta) (h(k-\rho k) - h(k-j)) \log n   - j \log (j/d_1)\leq  \log C.
\end{align}}
This immediately holds for $j < \rho k$, since in this case there exists $C_1 = C_1(d_1)>0$ such that
\begin{align*}
\text{LHS in} \ \eqref{118} \leq -j\log (j/d_1) \leq C_1.
\end{align*} 
Thus, it only remains to verify the same for  $j \geq \rho k$.
Towards this end, we regard the LHS in \eqref{118} as a function of $j$.
Let $$g(x) := - (1-\eta) h(k-x)  \log n + (1-\eta) h(k-\rho k)  \log n  - x \log (x/d_1).$$
Then,
\begin{align} \label{g'}
g'(x) =  (1-\eta) h'(k-x)     \log n  - \log (x/d_1) -1.
\end{align} 
Recall that by the upper bound for $h'$ in \eqref{h'(y)}, for large enough $k$, $h'(k-\rho k)  <(a+\delta) \frac{2}{3}\frac{1}{(k-\rho k)^{1/3}} $. {Hence, for small  $\rho>0$ and   large  enough $k$,
\begin{align*}
g'(\rho k) & = (1-\eta) h'(k-\rho k)     \log n  - \log (\rho k /d_1) -1 \\
&\leq (1-\eta) (a+\delta) \frac{2}{3}\frac{1}{(k-\rho k)^{1/3}}  \log n  -  \log ( \rho k/d_1 ) -1 \\
& \leq \frac{2}{3} \Big(\frac{k}{k-\rho k}\Big)^{1/3} \log k - \log ( \rho k/d_1 ) -1  \\
&= \Big( \frac{2}{3}\Big(\frac{1}{1-\rho}\Big)^{1/3} - 1\Big) \log k - \log (\rho/d_1)  - 1 < 0.
\end{align*}
Note that  we used  the condition $(a+\delta)\log n <   k^{1/3} \log k$ in the second inequality above.}

\noindent
Using this fact, we  next show that
\begin{align} \label{126}
\max_{x\in [\rho k, k-10]} g(x)  =  \max(g(\rho k),g(k-10)).
\end{align}
By \eqref{third},  $h'(k-x)$ is convex  on $(\rho k, k-10)$, which also implies the convexity of $g'(x)$ (see the expression for $g'$ in \eqref{g'}). Since $g'(\rho k)<0$, there are only two possible scenarios:  either $g'<0$ on {$(\rho k, k-10)$}, or $g'<0$ on $(\rho k,x_0)$ and $g'>0$ on $(x_0, k-10)$ for some  $x_0\in (\rho k, k-10)$. This implies   \eqref{126}.

Thus, in order to show \eqref{118},   it suffices to consider the two cases, $j \geq k-10$ and $j=\rho k$.
}
In the case $j=k-z$ with $0\leq z\leq 10$,   the LHS of  \eqref{118} can be bounded as
{\begin{align*}
 (1-\eta) &(h(k-\rho k) - h(k-j)) \log n   - j \log (j/d_1) \\
 &\leq 
  (1-\eta)   h(k-\rho k)    \log n   - (k-10) \log ((k-10)/d_1)  \\
  & \leq   h(k) \log n - k\log k + Ck \leq a  k^{2/3} \log n - k\log k+Ck \leq 0.
\end{align*}
Here, we used the condition $k^{1/3} \log k> (a+ \delta) \log n$ to deduce the last inequality.}

\noindent
Whereas, in the case $j=\rho k$, \eqref{118} immediately holds, completing the verification of the claim.

\qed

We now proceed towards proving Theorem \ref{theorem 2}. 

\section{Proof of Theorem \ref{theorem 2}} \label{secstructure}
{We start by recalling the events from Theorem \ref{theorem 2}:
\begin{align*}
\mathcal{D}_{\e}:=\{  \text{There exist at least}  \    (1-\varepsilon) k \  \text{vertex-disjoint triangles} \}
\end{align*}
and
\begin{align*}
\mathcal{C}_{\e}:= \{ \text{There exists} \ V'\subseteq V_n  \  \text{such that} \ |V'| \leq   (1+ \varepsilon) 6^{1/3}k^{1/3}, \Delta(V') \geq  (1-\varepsilon ) k \}.
\end{align*}
 We now define an auxiliary event $\mathcal{T}_\e$, as follows:
\begin{align*}
\mathcal{T}_\e:= \{ \text{There exists a triangle-induced subgraph spanned by at least}  \ (1-2\e)k \  \text{many triangles}\}.
\end{align*}
}

Note that $\Delta(G) \leq \frac{1}{6}|G|^3$ for any graph $G$, and it is not too difficult to show that an ``almost" equality holds if and only if $G$ is ``close" to a clique for appropriate formulation of the terms in quotes. Hence,  $\mathcal{C}_\e$ implies that  $V'$ looks like a clique containing almost $k$ many triangles.}

The proof of Theorem \ref{theorem 2} has broadly two parts. The first part is proving that conditional on the event $\{N\ge k\},$ the event $\cD_\e \cup \cT_\e$ holds with high probability (see Proposition \ref{prop}).  The next part proves that  conditionally on $\cT_\e \cap \{N\ge k\}$,  the event $\cC_{12\e^{1/4}}$ holds  with high probability (see Proposition \ref{prop2}). {By these two parts, with high probability conditionally on $\{N\geq k\}$, the event $\cD_\e \cup \cC_{12\e^{1/4}}$ holds. By the monotonicity of the events $\cD_\e$ and $\cC_\e$, in $\e$,  one can conclude that  conditionally on $\{N\ge k\},$  $\cD_\e \cup \cC_\e$ holds with high probability.

Statements \eqref{012} and \eqref{013} follow from Theorem \ref{theorem 1}, along with the fact that   for small enough $\e>0$, it is unlikely that $\cD_\e \cap \cC_\e$ happens  given $\{N\ge k\}$. Validity of   \eqref{012} and \eqref{013} for any $\e>0$ follows again by the monotonicity of the events $\cD_\e$ and $\cC_\e$ in $\e$.}

\begin{proposition} \label{prop}
The following holds for sufficiently small $\e>0$.
For any $k$ such that $k\rightarrow \infty$ and $k\leq n^{\frac{1}{10}-\mu}$,
\begin{align*}
\lim_{n\rightarrow \infty} \mathbb{P}(\cD_\e \cup \cT_\e | N\geq  k )  =  1.
\end{align*}
\end{proposition}

{\begin{proposition}\label{prop2}
The following holds for sufficiently small $\e>0$.
For any $k$ such that $k\rightarrow \infty$ and $k\leq n^{\frac{1}{10}-\mu}$,
\begin{align*}
\lim_{n\rightarrow \infty} \mathbb{P}( \cT_\e \cap  \cC^c_{12\e^{1/4}} |  N\geq k )  =  0.
\end{align*}
\end{proposition}

We now finish the proof of Theorem \ref{theorem 2} first.

\begin{proof}[Proof of Theorem \ref{theorem 2}] 
By the monotonicity of the events $\cD_\e$ and  $\cC_\e$ in $\e,$ it suffices to prove them for small enough $\e,$ an assumption that is particularly convenient for some of our arguments. 
 By the above two propositions, {for any small enough $\e>0$},
\begin{align*}
\lim_{n\rightarrow \infty} \mathbb{P}(\cD_\e \cup  \cC_{12\e^{1/4}} |  N \geq k )  =  1.
\end{align*} 
{Observing that $\cD_\e$ and $\cC_\e$ are decreasing as $\e\rightarrow 0$, by the arbitrariness of $\e>0$,}
\begin{align}\label{0112}
\lim_{n\rightarrow \infty} \mathbb{P}(\cD_\e \cup  \cC_\e |  N\geq k )  =  1.
\end{align} 

\noindent
To obtain \eqref{012} and \eqref{013}, we first verify that for sufficiently small $\e>0$,
\begin{align} \label{430}
\lim_{n\rightarrow \infty} \mathbb{P}( \cD_\e \cap  \cC_\e |  N \geq k )  =  0.
\end{align} 
By \eqref{holder}, $\cC_\e$ implies the existence of a subset $V'$ such that {$|V'| \leq (1+\e) 6^{1/3} k^{1/3}=:v$ and   $|E(V')|\geq  a (1-\e)^{2/3} k^{2/3}=:e$ (where, recall that the value of $a$ is specified to be $ (\frac{3}{\sqrt{2}})^{2/3}$).} By a first moment bound, there is $\gamma = \gamma(\e)>0$ with $\lim_{\e \rightarrow 0} \gamma = 0$ such that for large enough $k$, the probability that such a subgraph exists is bounded by
{
\begin{align} \label{431}
\binom{n}{v}  \binom{\binom{v}{2}}{e} \left ( \frac{d}{n} \right )^{e} 
\leq 
n^{v} 2^{v^2}  \left ( \frac{d}{n} \right )^{e}  
\leq n^{ - (a - \gamma)  k^{2/3}}.
\end{align} 
}
{Recall that  $\cD_\e$ implies the existence of {at least}  $(1-\e)k$ disjoint triangles, and under the event $\cC_\e$, there are at most $|V'|  = O(k^{1/3})$ {vertex-disjoint} triangles sharing some vertex in $V'$. Hence,  for large enough $k$, $\cD_\e \cap  \cC_\e$ implies the disjoint occurrence of the subgraph $V'$ and at least $(1-2\e)k$ disjoint triangles.}
Thus, by Remark \ref{triangle upper} and  BK inequality, 
\begin{align*}
 \mathbb{P}(\cD_\e \cap  \cC_\e )  \leq n^{ - (a - \gamma)  k^{2/3}} \cdot e^{-(1-2\e) k\log k +2 c'k} =: S.
\end{align*}
Relying on Proposition \ref{lower}, for small enough $\e>0$, it follows that the quantity $S$  is negligible compared to   $\mathbb{P}(N \geq k)$. To see this,   observe that $S  = o(1) e^{- k\log k  - ck}$ is equivalent to
\begin{align*}
 - (a-\gamma ) k^{2/3} \log n + 2\e k \log k  + 2c'k \rightarrow  - \infty,
\end{align*}
which holds  for small enough $\e>0$ when $k^{1/3} \log k \leq a\log n$. Further,  $S  = o(1) n^{-(1+\e) ak^{2/3}}$ is equivalent to
\begin{align*}
  (\gamma +  a\e ) k^{2/3} \log n - (1- 2\e) k \log k  + 2c'k \rightarrow  - \infty,
\end{align*}
which holds    for small enough $\e>0$  when $k^{1/3} \log k \geq a\log n$.   
Therefore,
we have  \eqref{430}.

\noindent
To conclude the proof, note that 
when $k^{1/3} \log k \leq (a-\delta)\log n$, by Lemma \ref{lower1}
and
\eqref{431},
\begin{align*}
\lim_{n\rightarrow \infty} \mathbb{P}(  \cC_\e |  N\geq k )  =  0,
\end{align*} 
whereas,
in the case $k^{1/3} \log k \geq (a+\delta)\log n$, by Lemma \ref{lower2} and Remark \ref{triangle upper},
\begin{align*}
\lim_{n\rightarrow \infty} \mathbb{P}(  \cD_\e |  N\geq k )  =  0.
\end{align*} 
{
This, along with \eqref{0112} and \eqref{430} finishes the proof of \eqref{012} and \eqref{013} for sufficiently small $\e>0$. Since events $\cD_\e$ and $\cC_\e$ are increasing in $\e$, this concludes the proof of \eqref{012} and \eqref{013}   for any $\e>0$.
}
\end{proof}

It remains to prove the propositions which we turn to next.

\begin{proof}[Proof of Proposition \ref{prop}]
We show that 
\begin{align*}
\mathbb{P}(\cD_\e^c \cap \cT_\e^c \cap  \{N \geq k\})  = o(1) \mathbb{P}(N\geq k).
\end{align*} 
The event $\cD_\e^c \cap \cT_\e^c \cap  \{N \geq k\}$ implies the  occurrence of the event $E_{\ell_1,\cdots,\ell_m}$ for some $1\leq \ell_1\leq \cdots \leq \ell_m \leq (1-2\e) k$  with $\ell_1+\cdots+\ell_m=k$   such that the number of 1s among the $\ell_i$s, denoted by $j$, is less than $(1-\e)k$.

\noindent
Thus,  for any $\eta>0$, for sufficiently large $n$,
{
\begin{align}   
\mathbb{P}(\cD_\e^c \cap \cT_\e^c \cap  \{N \geq k\})  &\leq \sum_{j=0}^{(1-\e)k}  \sum_{{\ell_{j+1} + \cdots + \ell_m  =  k-j}\atop{2\leq  \ell_{j+1}\leq \cdots\leq  \ell_m\leq (1-2\e)k}}  \mathbb{P} (E_{\ell_1,\cdots,\ell_m}) \nonumber \\
& \overset{\eqref{E}}{\leq}  \sum_{j=0}^{(1-\e)k}  \sum_{{\ell_{j+1} + \cdots + \ell_m  =  k-j}\atop{2\leq  \ell_{j+1}\leq \cdots\leq  \ell_m\leq (1-2\e)k}}   \frac{1}{j!} \Big(\frac{d^3}{6}\Big)^j  \Big(\frac{1}{n}\Big) ^{(1-\eta) (h(\ell_{j+1}) + \cdots + h(\ell_m)) } .
\end{align}
}
{
In order to prove that this is negligible compared to  $ \mathbb{P}(N\geq k)$, we  first show that the summation above is negligible for $\e k \leq j \leq (1-\e ) k$ (Step 1 below). Then, we verify that  the restriction $\ell_m \leq (1-2\e)k$ implies that the remaining summation, i.e. over parts $j\leq \e k,$ is also negligible (Step 2). }

\noindent
\textbf{Step 1.} We prove that
\begin{align}
 \sum_{j=\e k}^{(1-\e)k}  \sum_{{\ell_{j+1} + \cdots + \ell_m  =  k-j}\atop{2\leq  \ell_{j+1}\leq \cdots\leq  \ell_m }}  \frac{1}{j!} \Big(\frac{d^3}{6}\Big)^j  \Big(\frac{1}{n}\Big) ^{(1-\eta) (h(\ell_{j+1}) + \cdots + h(\ell_m)) }  = o(1) \mathbb{P}(N\geq k).
\end{align}
\noindent
{ By  \eqref{110}}, the above sum can be bounded by
\begin{align}  \label{510}
  \sum_{j=\varepsilon k}^{(1-\varepsilon) k}   \frac{1}{j!} \Big(\frac{d^3}{6}\Big)^j   e^{\alpha\sqrt{k-j}}    \Big(  \frac{1}{n}\Big)^{(1-\eta) h(k-j -10^5) }.
\end{align}
{Recalling $d_1= \frac{ed^3}{6}$ and by Stirling's formula \eqref{stirling},}
each term above is bounded by
{\begin{align} \label{511}
d_1^j e^{\alpha\sqrt{k-j}}     e^{-( j\log j + (1-\eta)  h(k-j -10^5)  \log n ) } .
\end{align}}
{Suppose that $j= z k$ with $\e\leq z\leq 1-\e$.  {Recalling that $h(k-j) \approx a(k-j)^{2/3}$ (see \eqref{h(y)} and \eqref{hy})}, for any $\iota>0$, for large  enough $k$,
\begin{align*}
h(k-j -10^5)  \geq   (1-0.1 \iota) h(k-j)  \geq (a-\iota) (k- j )^{2/3} =  (a-\iota) (1-z)^{2/3} k^{2/3}.
\end{align*}
}
{This implies} that for some {constant $c_1 >0$,}
{\begin{align*}
j\log j + (1-\eta)h(k-j -10^5)\log n& \geq  z k \log z +  z  k \log k+ (1-\eta)  (a-\iota) (1-z )^{2/3}  k^{2/3} \log n  \\
& \geq -c_1k + \Big(z + (1-\eta) (a-\iota)   \frac{\log n}{k^{1/3} \log k }   (1-z )^{2/3}  \Big )  k\log k  .
\end{align*}}
{In order to lower bound the second term, let us examine
the concave function $z \mapsto  z +  \chi (1-z )^{2/3}$ on $[0,1],$ parametrized by $\chi>0$. We will show that
 if $\e>0$ is small enough, then for any $\e \leq z\leq 1-\e$ and $\chi \in (0,1)$,
\begin{align} \label{512}
 z +  \chi (1-z )^{2/3}  \geq (1+0.01\e )^{2/3} \chi.
\end{align} }
Let us first conclude the proof of Step 1 assuming \eqref{512}. Using this, there exists $\gamma = \gamma(\e)>0$ such that the following holds for small enough $\iota,\eta,\e>0$ and any $\e\leq z\leq 1-\e$.\\
\noindent
$\bullet$  when $k^{1/3}\log k \leq  a \log n$: 
  \begin{align*}
 z +  (1-\eta)   (a-\iota)\frac{\log n}{k^{1/3} \log k}   (1-z )^{2/3}    \geq  1 + \gamma .
\end{align*} 
To see this, note that for any  $\kappa>0$,   $ (1-\eta)      (a-\iota)\frac{\log n}{k^{1/3} \log k} \geq 1  -  \kappa $ for small enough $\eta,\iota>0$. Thus, 
{\begin{align*}
 z +  (1-\eta)   (a-\iota)\frac{\log n}{k^{1/3} \log k}   (1-z )^{2/3}   & \geq  z+ (1-\kappa) (1-z )^{2/3} \\
 & \geq  (1+0.01\e )^{2/3}  (1-\kappa) \geq 1+0.001\e
\end{align*} 
for small enough $ \kappa>0$,
where the second last inequality follows by setting $\chi = 1-\kappa$ in \eqref{512}.} \\

\noindent
Therefore, for large enough $k$, using \eqref{511}, the quantity \eqref{510} in this case is bounded by
 \begin{align} \label{521}
 k d_1^k e^{\alpha\sqrt{k}}   e^{c_1k}   e^{  - (1+  \gamma) k\log k}   \leq  e^{  - (1+   \gamma/ 2) k\log k}  .
 \end{align}

\noindent
$\bullet$  {when $k^{1/3}\log k \geq a \log n$:   setting $\chi =   (1-\eta)   (a-\iota)  \frac{\log n}{k^{1/3} \log k} \in (0,1)$ in \eqref{512}, taking $\gamma>0$ to be $1+2\gamma = (1+0.01\e)^{2/3}$,}
 \begin{align*}
 z +  (1-\eta)   (a-\iota)  \frac{\log n}{k^{1/3} \log k}    (1-z )^{2/3}    &\geq    (1+2 \gamma) (1-\eta)(a-\iota) \frac{\log n}{k^{1/3} \log k}  \\
 &\geq (1+\gamma)  a\frac{\log n}{k^{1/3} \log k}.
\end{align*} 
Hence, for large enough $k$, in this case the quantity \eqref{510} is bounded by
 \begin{align} \label{522}
  k d_1^k e^{\alpha\sqrt{k}} e^{c_1k} e^{- a (1+\gamma) k^{2/3} \log n  }  \leq n^{-a (1+ \gamma/2 ) k^{2/3}}.
\end{align}   
By Proposition \ref{lower},
quantities \eqref{521} and \eqref{522} are negligible compared to $\mathbb{P}(N\geq k)$ for large enough $k$.

Thus to conclude Step 1 it remains to prove \eqref{512}. By concavity of $z\mapsto z +  \chi (1-z )^{2/3}$, it suffices to consider only two cases $z=\e$ and $z=1-\e$. In the case $z=\e$, using $\chi <1$ and mean value theorem, for small enough $\e>0$,
\begin{align*}
{\chi (  (1+0.01\e )^{2/3} - (1-\e)^{2/3})  \leq 1.01\e \cdot   \frac{2}{3} \cdot 1.1  \leq \e.}
\end{align*}
which after rearranging yields \eqref{512}.
In the case $z=1-\e$, since $\chi\in (0,1)$ it suffices to prove  $ (1+0.01\e )^{2/3}  - \e^{2/3}  \leq   1-\e$. By mean value theorem, for small enough $\e>0$,
\begin{align*}
  (1+0.01\e )^{2/3}  - 1 \leq 0.01\e \cdot  \frac{2}{3}\cdot  1.1 \leq \Big(\frac{1}{\e^{1/3}}-1 \Big )\e  = \e^{2/3}-\e.
\end{align*}

\noindent 
 \textbf{Step 2.}
We next show that  
\begin{align}  \label{120}
\sum_{j=0}^{\varepsilon k}  \sum_{{\ell_{j+1} + \cdots + \ell_m  =  k-j}\atop {2\leq \ell_{j+1}\leq \cdots \leq  \ell_m < (1-2\varepsilon ) k}}   \frac{1}{j!} \Big(\frac{d^3}{6}\Big)^j  \Big(\frac{1}{n}\Big) ^{(1-\eta)(h(\ell_{1}) + \cdots + h(\ell_m)) }   = o(1)\mathbb{P}(N\geq k).
\end{align}  
We claim that for sufficiently small $\e>0$,   for large enough $k$ and  any $ \ell_{j+1} + \cdots + \ell_m  =  k-j, \ell_{j+1}\leq \cdots \leq \ell_m < (1-2\varepsilon ) k$ with $j\leq \e k$,
\begin{align} \label{115}
h(\ell_{j+1}) + \cdots + h(\ell_m)  \geq h(k) + 3   \e^2  k^{2/3}.
\end{align}
{
We first conclude the proof of \eqref{120} assuming this claim.} Using  the above {along with Stirling's formula \eqref{stirling},} for sufficiently small $\eta>0$, the quantity \eqref{120} is  bounded by
{\begin{align} 
  \sum_{j=0}^{\varepsilon k}   \frac{d_1^j}{j^j}   e^{\alpha \sqrt{k}} \Big(  \frac{1}{n}\Big) ^{  (1-\eta)(h(k)+ 3 \varepsilon^{2} k^{2/3}  )}    \leq   C \e k  e^{\alpha\sqrt{k}} \Big(  \frac{1}{n}\Big) ^{ h(k)+ 2 \varepsilon^{2} k^{2/3}  }   \leq  \Big(  \frac{1}{n}\Big) ^{  h(k)+  \varepsilon^{2} k^{2/3} } ,
\end{align}}
{where the first inequality uses $ \frac{d_1^j}{j^j}\leq C$ for any $j\geq 0$ (we set $0^0:=1$).}
Thus, along with the lower bound $\mathbb{P}(N\geq k)$ (Proposition \ref{lower}) concludes the proof of \eqref{120}.

\noindent
Now, let us verify the claim \eqref{115}.
For any $x,y\in [10,(1-2\e)k]$, by \eqref{eq:h_subadd_10}, it holds that
\begin{align*}
h(x)+h(y)  \geq h(10) + h(x+y-10),\qquad { \text{if}} \quad  x+y-10\leq (1-2\e)k
\end{align*}
and by a similar reasoning we also have that
\begin{align*}
 h(x)+h(y)  \geq h(x+y-(1-2\e)k) + h((1-2\e)k),\qquad { \text{if}} \quad  x+y-10 \geq  (1-2\e)k.
\end{align*}

{
We apply this repeatedly to  $\ell_{j+1},\cdots,\ell_m$ until there is at most one element  strictly between $10$ and $(1-2\e)k$. At the end of this procedure, let $i$ be the number of elements less than or equal to 10. Then, we obtain the following two cases, depending on whether $ (1-2\e)k$ appears as a term or not (since this procedure preserves the summation, the term $ (1-2\e)k$ appears at most once for small $\e>0$): 

\noindent
Case 1. $ h(    (1-2\varepsilon ) k)$ does not appear:
\begin{align} \label{121}
h(\ell_{j+1}) + \cdots + h(\ell_m)  \geq i+ h( b_1 )
\end{align}
  with $ (1-2 \varepsilon) k \geq   b_1 \geq k-j-10i$. The lower bound on $b_1$ is a consequence of the facts that the above procedure preserves the summation to be $k-j$ and that $h$ is monotone. \\ 

\noindent
Case 2.  $ h(    (1-2\varepsilon ) k) $ appears:
\begin{align} \label{122}
h(\ell_{j+1}) + \cdots + h(\ell_m)  \geq i + h(b_2) +  h(    (1-2\varepsilon ) k)
\end{align}
  with
 $b _2 \geq  k-j- (1-2\varepsilon ) k - 10i = 2\e k-j-10i$.  The reason for the  lower bound on $b_2$ is similar as above, along with a fact that the term $(1-2\e)k$ is obtained at the end of  procedure.}

 We  first consider the former case. Since  $ (1-2 \varepsilon) k \geq   b_1 \geq k-j-10i$ and $j\leq \e k$, we have  $10i \geq   \e k$. Thus,  for large enough $k$,
\begin{align*}
i + h(b_1) \geq i\geq  \frac{1}{10}\e k  \overset{\eqref{h(y)}}{ \geq } h(k) + 3   \e^2  k^{2/3}.
\end{align*}
{Note that the last inequality holds since LHS and RHS are of order $k$ and $k^{2/3}$ respectively.
}

\noindent
We next consider the latter case  \eqref{122}.  Since  $2\e k-j \geq 0$ (recall that $j \leq \varepsilon k$),
by Lemma \ref{lemma 3.0},
\begin{align*}
i+h( b_2) + h((1-2\e )k) &\geq i+h(2   \varepsilon k - j  - 10i) +  h((1-2\e )k) 
\\ & \geq  h(2\varepsilon k-j) - h(10^6)   +  h((1-2\e )k) .
\end{align*}
Thus, in order to verify \eqref{115},
it suffices to show that for sufficiently small $\varepsilon >0$, for large  enough $k$,
\begin{align}   \label{123}
h(2\varepsilon k-j) + h (  (1-2\varepsilon ) k )   \geq h(k) +3  \e^2  k^{2/3} + h(10^6).
\end{align}
In fact, by mean value theorem, for large enough $k$,
{\begin{align*}
h(k) -  h (  (1-2\varepsilon ) k )  \leq  2\varepsilon k  h'(  (1-2\varepsilon ) k )\overset{ \eqref{h'(y)}}{ \leq} 2\varepsilon k \cdot 1.01 \cdot  a \frac{2}{3} ( k -2\varepsilon k)^{-1/3}.
\end{align*}}Since $j \leq \varepsilon k$,  by \eqref{hy},  for large enough $k$,
\begin{align*}
h(2\varepsilon k-j) \geq h(\varepsilon k)\geq \frac{1}{2}a (\varepsilon k)^{2/3} + h(10^6).
\end{align*}
{
By the above two displays, to prove \eqref{123}, it suffices to check}
\begin{align*}
 \frac{1}{2}a (\varepsilon k)^{2/3} \geq 2\varepsilon k \cdot 1.01 \cdot  a \frac{2}{3} ( k -2\varepsilon k)^{-1/3} + 3\e^2 k^{2/3}.
\end{align*}
Since the coefficients of $k^{2/3}$ in LHS and RHS are of order $\e^{2/3}$ and $\e$ as $\e\rightarrow 0$ respectively, this inequality holds  for sufficiently small $\e>0$.

\end{proof}

We now move on to the proof of Proposition \ref{prop2}.

\begin{proof}[Proof of Proposition \ref{prop2}]

{Under the event  $\cT_\e$, let $G'$ be a $\TISG$   spanned by  $\tilde{k} := \left \lceil{ (1-2 \varepsilon )k}\right \rceil
 $ triangles. We will  prove that  with high probability conditionally on $\{N\geq k\}$,  $G'$ contains a subgraph which looks like a clique and contains at least $(1-12\e^{1/4})k$ triangles.  
{The argument has three parts (carried out in the sequel in Steps 1,2 and 3 respectively). {First,  we show that with high probability conditioned on $\{N\geq k\}$, the number of tree-excess edges in  $G'$  is at most of order $k^{2/3}$ (Step 1). Hence by Lemma \ref{lemma 4.3}, $G'$ contains a `dense' subgraph $G''$ of  size $O(k^{1/3})$ which  contains close to $k$ triangles.    Next, we deduce that with high probability conditioned on $\{N\geq k\}$, any subgraph of size $O(k^{1/3})$ has  `almost' less than $h(k)$ edges (Step 2). This implies that $G''$ almost attains the equality in \eqref{holder} and thus as an application of Proposition \ref{equality holder}, $G''$ contains a subgraph which is close to a clique and contains almost $k$ triangles. This is done in Step 3.}}

\noindent
\textbf{Step 1.}
We prove that there exists a large  constant $b>0$ such that following holds: With high probability conditionally on $\{N\geq k\}$, any $\TISG$  $H$ spanned by $\tilde{k}$ triangles satisfies
{\begin{align}\label{401}
 |E(H)| - |V(H)| \leq   b \tilde{k}^{2/3} -1.
\end{align}} 
As in the proof of Lemma \ref{lemma 3.3},  let
$F_{\tilde{k},v,e}$ be an event that there exists  a $\TISG$ spanned by  $\tilde{k}$ triangles with $v$  vertices and  $e$ edges.  Then,
\begin{align*}
    \mathbb{P}&( \,
    \exists\, \TISG \, \  H \ \text{spanned by} \ \tilde{k} \  \text{triangles with} \  |E(H)| - |V(H)| \geq   b \tilde{k}^{2/3} -1 ) \\
    &   \leq  \sum_{   {e\geq v+  b \tilde{k}^{2/3} -1 } \atop {e,v\leq 3\tilde{k}} }    \mathbb{P}(F_{\tilde k,v,e} )   \overset{\eqref{332}}{\leq}   \sum_{   {e\geq v+   b \tilde{k}^{2/3}-1 } \atop {e,v\leq 3\tilde{k}} }    \Big  ( \frac{d^5v^{10}}{n} \Big )^{e-v}  d^5v^{10}  \\
    & \leq  (3\tilde{k})^2 \Big ( \frac{d^5(3\tilde{k})^{10}}{n}\Big)^{ b\tilde{k}^{2/3}-1}  d^5 (3\tilde{k})^{10}  \\
   & \leq    { C\tilde{k}^{12} \cdot  \frac{(C\tilde{k})^{10b\tilde{k}^{2/3}}}{n^{b\tilde{k}^{2/3}-1}  } }   \leq  Ck^{12} \cdot  \frac{(Ck)^{10bk^{2/3}}}{n^{b\tilde{k}^{2/3}-1}  }. 
\end{align*}
We show that this is bounded by    $(\frac{1}{n})^{h(k)+ \e b  k^{2/3}} $
  for small enough $\xi,\e>0$.  Once we show this, by  the lower bound on $\mathbb{P}(N\geq k)$ in   Lemma \ref{lower2},  \eqref{401} follows.
 
\noindent
Rearranging, we need to establish
  \begin{align} \label{406}
C^{10bk^{2/3}+1}  \cdot  k^{12 + 10bk^{2/3}} \leq n^{b\tilde{k}^{2/3} -1 - h(k)- \e b  k^{2/3}  }.
  \end{align}
By the condition $ k \leq n^{\frac{1}{10}-\mu}$, for large enough $n$, the LHS  is bounded by
\begin{align*}
n^{b \mu k^{2/3}}\cdot  n^{(1-10\mu) bk^{2/3} +2} = n^{ (1-9\mu) b k^{2/3}+2}.
\end{align*}
Whereas,  for small enough  $\e>0$ and large enough $b>0$, the exponent of $n$ in the RHS of \eqref{406} is
{
\begin{align*}
 b\tilde{k}^{2/3} -1- h(k)- \e b  k^{2/3} & \overset{\eqref{h(y)}}{  \geq }   ( (1-2\e)^{2/3}  - \e) b k^{2/3}  - ak^{2/3} -1\\
&\geq ((1-\mu)b  - a) k^{2/3} >  (1-9\mu) b k^{2/3}+2,
\end{align*}
where the last inequality  uses the fact that  $(1-\mu)b  - a   >  (1-9\mu) b $ for  large enough $b$.
We thus obtain \eqref{406}.
}

}

\noindent
\textbf{Step 2.} 
 For any {constants $c_1,c_2>0$}, with high probability {conditionally on $\{N\geq k\}$}, any subgraph $H$ (need not be a $\TISG$ {or a connected graph}) with $|V(H)| < c_1k^{1/3}$ satisfies
{\begin{align} \label{402}
|E(H)| <  h(k) + c_2 k^{2/3}
\end{align}}
for  large enough $n\geq n_0(c_1,c_2)$ and $k\geq k_0(c_1,c_2)$ (a useful consideration is the case when $c_1$ is large and $c_2$ is small).
 In fact, for large enough $n$ and  $k$, the probability that there exists a subgraph $H$ with $|V(H)| < c_1k^{1/3}$  and  $|E(H)| \geq h(k) + c_2 k^{2/3} $
 is bounded by
 \begin{align*}
n^{c_1 k^{1/3}} 2^{c_1^2 k^{2/3}}  \Big(\frac{d}{n}\Big)^{h(k) + c_2 k^{2/3} } \leq \Big(\frac{1}{n}\Big)^{h(k) + \frac{c_2}{2} k^{2/3}}.
 \end{align*}
{By the lower bound on $\mathbb{P}(N\geq k)$ in  Lemma \ref{lower2}}, \eqref{402} follows.\\

\noindent
{\textbf{Step 3.} Recall that $G'$ is a $\TISG$ spanned by  $\tilde{k}$   triangles.}
 By \eqref{401}, for a  large enough  constant $b>0$, with high probability {conditionally on $\{N\geq k\}$}, {
 \begin{align*}
  |E(G')| - |V(G')| \leq     b \tilde{k}^{2/3}-1   .
\end{align*} }
Thus, setting $\xi>0$ to be such that $b= \frac{1}{2} \xi^{-1/2} $, $G'$ satisfies the hypotheses of Lemma \ref{lemma 4.3}, on the application of which, we conclude that  for sufficiently large $k$, $G'$ contains a dense subgraph $G''$ such that
\begin{align} \label{403}
|V(G'')| \leq  \xi^{-3/2}\tilde{k}^{1/3}  \leq  \xi^{-3/2} (1-\e)^{1/3} k^{1/3}  
\end{align}
and ({by taking large enough $b$, i.e. small enough $\xi>0$})
\begin{align} \label{404}
  \Delta(G'') \geq  (1-2\xi^{1/2})\tilde{k}  \geq     (1-3\e) k.
\end{align}
{Note that Lemma \ref{lemma 4.3} also implies a  lower bound for $|E(G'')|$ from \eqref{333}. However for our purposes, we will instead be interested in a sharp upper bound for $|E(G'')|$, which along with \eqref{404} will be used to finish the argument.} Towards this, 
by \eqref{402} and \eqref{403},    with high probability (conditionally),
\begin{align} \label{405}
 |E(G'')| \leq  h(k) + a \e  k^{2/3} \leq  a(1+\e) k^{2/3}.
\end{align}  
Thus
\eqref{404} and \eqref{405} imply that $G''$ almost attains an equality in Lemma \ref{lemmaholder}. This allows us to invoke Proposition \ref{equality holder}.  Applying the latter with $6k$ and $3\e$ in place of $k$ and $\e$ respectively, allows us to conclude that $G''$ contains a subgraph of size at most  $(1+12\e^{1/4})6^{1/3} k^{1/3}$ with at least $(1-12\e^{1/4})k$ triangles. This concludes the proof.

\end{proof}

 \bibliography{ldp_ref}
\bibliographystyle{plain}

\end{document}